\documentclass[11pt]{article}
\usepackage{authblk}
\usepackage{amsmath,amssymb,mathrsfs,centernot}
\usepackage{enumitem}
\usepackage{verbatim}
\usepackage{color}
\usepackage{float}
\usepackage[margin=1in]{geometry}
\usepackage{setspace}
\usepackage{graphics,graphicx}
\usepackage{quiver}
\usepackage{pifont}
\usepackage{soul}

\newcommand{\fstar}{\text{\ding{72}}}
\newcommand{\mbst}{\mathbf{mb}^\fstar}
\newcommand{\amst}{\mathsf{amst}}
\newcommand{\mstr}{\mathfrak{M}}
\newcommand{\Mod}{\mathsf{Mod}}
\newcommand{\lang}{\mathcal{L}}
\newcommand{\limp}{\longrightarrow}

\newcommand{\pow}{\mathcal{P}}

\newcommand{\NN}{\mathbb{N}}

\newcommand{\gecq}{\mathsf{gECQ\text{-}sat}}
\newcommand{\secq}{\mathsf{sECQ\text{-}sat}}
\newcommand{\specq}{\mathsf{spECQ\text{-}sat}}
\newcommand{\pfecq}{\mathsf{pfECQ\text{-}sat}}
\newcommand{\fgecq}{\mathsf{gECQ\text{-}finsat}}
\newcommand{\fsecq}{\mathsf{sECQ\text{-}finsat}}
\newcommand{\fspecq}{\mathsf{spECQ\text{-}finsat}}
\newcommand{\fpfecq}{\mathsf{pfECQ\text{-}finsat}}

\usepackage{amsthm}
\theoremstyle{definition}
\newtheorem{thm}{Theorem}[section]

\newtheorem{cor}[thm]{Corollary}
\newtheorem{dfn}[thm]{Definition}
\newtheorem{rem}[thm]{Remark}
\newtheorem{exa}[thm]{Example}

\usepackage[backref=page,colorlinks=true,allcolors=blue]{hyperref}

  \title{Generalized Explosion Principles: A Semantic Perspective}
  \author[1]{Sankha S. Basu}
  \author[2]{Sayantan Roy}
  \date{November 2, 2025}
  \affil[1]{Department of Mathematics\\
  Indraprastha Institute of Information Technology-Delhi\\
  New Delhi, India.}
  \affil[2]{International Laboratory for Logic, Linguistics and Formal Philosophy\\National Research University Higher School of Economics\\Moscow\\Russia}

\begin{document}
\setstcolor{magenta}
\maketitle

\begin{abstract}
This article is motivated by the fact that there is a distinction between the descriptions of logical explosion from syntactic and semantic points of view. The discussion is illustrated using the concept of \emph{abstract model structures} and the notions of satisfiability and finite satisfiability in these structures. Various principles of explosion have been described in terms of unsatisfiability or finite unsatisfiability. The semantic analogues of the principles of explosion introduced in \cite{BasuRoy2024} have also been considered among these. The article also studies the characterizations of and the interconnections between these new principles of explosion.
\end{abstract}

\tableofcontents

\section{Prologue}
This article continues the discussion on generalizing the notion of logical explosion, and hence, paraconsistency as failure of a principle of (logical) explosion, that was started in \cite{BasuRoy2022} and later continued in \cite{BasuRoy2024}. As in these two articles, we use the framework of logical structures from the field of universal logic (see \cite{Beziau1994, Beziau2006}). A \emph{logical structure} is a pair $(\lang,\vdash)$, where $\lang$ is a set and $\vdash\,\subseteq\pow(\lang)\times\lang$ ($\pow(\lang)$ denotes the power set of $\lang$). Given a logical structure $(\lang,\vdash)$ and $\Gamma\subseteq\lang$, $C_{\vdash}(\Gamma)=\{\alpha\in \lang\mid\,\Gamma\vdash \alpha\}$. $C_\vdash$ can be seen as an operator from $\pow(\lang)$ to $\pow(\lang)$. Thus, given a relation $\vdash\,\subseteq\pow(\lang)\times\lang$, there is a corresponding operator $C_{\vdash}:\pow(\lang)\to\pow(\lang)$, where, for any $\Gamma\in\pow(\lang)$, $C_{\vdash}(\Gamma)$ is as described above. Conversely, given an operator $C:\pow(\lang)\to\pow(\lang)$, we can define a relation $\vdash\,\subseteq\pow(\lang)\times\lang$ such that $C=C_{\vdash}$ as follows. For all $\Gamma\cup\{\alpha\}\subseteq\lang$, $\Gamma\vdash\alpha$ iff $\alpha\in C(\Gamma)$. This observation allows us to move freely back and forth between a relation $\vdash$ and its corresponding operator $C_{\vdash}$, and on some occasions, define a $\vdash\,\subseteq\pow(\lang)\times\lang$ such that $C_{\vdash}=C$, for some operator $C:\pow(\lang)\to\pow(\lang)$ in this very sense. Finally, a \emph{logic} is a special case of logical structure where $\lang$ is an algebra (viz., an algebra of formulas). 

In this article, we consider a special type of logical structures, viz., \emph{model-theoretic logical structures} (see Definition \ref{dfn:mlogstr}). This is based on the concept of \emph{abstract model structures} introduced in \cite{RoyBasuChakraborty2025}.

The motivation behind this work stems from the observation that there is a fundamental difference between the syntactic and semantic treatments of logical explosion. Suppose $(\lang,\vdash)$ is a Hilbert-style presentation of classical propositional logic (CPL). Then, for any $\alpha,\beta\in\lang$, there is a derivation of $\beta$ from $\{\alpha,\neg\alpha\}$ (`$\neg$' denotes the classical negation). Thus, we have the most common rule of explosion, \emph{ex contradictione sequitur quodlibet (ECQ)}: $\{\alpha,\neg\alpha\}\vdash\beta$ for any $\alpha,\beta\in\lang$. Now, once CPL is interpreted via valuations from $\lang$ to $\{0,1\}$, the universe of the 2-element Boolean algebra, ECQ is a valid rule only because of the absence of any valuation $v$ such that $v(\alpha)=v(\neg\alpha)=1$. 

The distinction between the syntactic and semantic paths to explosion or paraconsistency can also be witnessed by considering the distinction or the lack thereof between the 3-valued strong Kleene logic $K_3$ and the logic of paradox $LP$. Although the connectives are interpreted identically in these two logics, $LP$ is paraconsistent, while $K_3$ is not. This difference is due to the interpretation of the non-classical third truth value - while in $K_3$ it is seen as indicating `neither true nor false', i.e., a \emph{gap}, in $LP$ it is taken as indicating `both true and false', i.e., a \emph{glut}. As a result, the third value is designated in $LP$ but not in $K_3$. This makes $LP$ paraconsistent and $K_3$ paracomplete, i.e., the law of excluded middle fails in $K_3$. A similar distinction exists between the paracomplete weak Kleene logic $WK$ and the paraconsistent weak Kleene logic $PWK$ (see \cite[Chapter 7]{priest2001}). In each of the four logics above, if any valuation maps a formula $\alpha$ to the non-classical value, then it also maps $\neg\alpha$ to the same value. Hence, in the case of $LP$ and $PWK$, there exist valuations that satisfy $\{\alpha,\neg\alpha\}$ in the sense that both $\alpha$ and $\neg\alpha$ are mapped to a designated value. This puts a check on logical explosion, and so, makes these logics paraconsistent. On the other hand, in the case of $K_3$ and $WK$, there is no valuation that satisfies $\{\alpha,\neg\alpha\}$. These two logics thus validate ECQ much like CPL. 

Thus, while the syntactic way of comprehending ECQ usually rests on the \textit{explosive nature} of the set $\{\alpha,\neg\alpha\}$, the semantic way rests on the \textit{unsatisfiability} of $\{\alpha,\neg\alpha\}$. In the case of CPL, this particular distinction is not of significant importance due to the soundness and completeness theorems. However, in the absence of either one of them, the distinction - and hence the difference - between these two becomes visible. In this article, we pursue a semantic\footnote{From the perspective of universal logic, it is perhaps better to say that our approach is `semantics-inspired' rather than `semantic.' This is because we are not considering anything that looks like `semantics' in the usual sense; we are merely using a certain kind of structure that looks similar to those often found in the semantic analysis of (concrete) logics. Nevertheless, for the sake of simplicity, we will continue to use the term `semantic' instead of its more precise non-equivalent alternative.} approach to logical explosion. The semantics is brought into the picture through the \emph{abstract model structures} or $\amst$s introduced in \cite{RoyBasuChakraborty2025}. The principles of explosion thus obtained are in terms of unsatisfiability or finite unsatisfiability. There are eight such principles obtained here. These are then connected to the principles of explosion introduced in \cite{BasuRoy2024}. We also investigate the interconnections between these eight principles and prove characterization theorems for them.

The article is structured as follows. In Section \ref{sec:amst}, the necessary background on abstract model structures and related concepts is discussed. This is followed by an analysis of finitary-ness of a model-theoretic logical structure and its connection to the compactness of the $\amst$ inducing it. In Section \ref{sec:sem}, we introduce the new principles of explosion. The section is split into three subsections: one for the principles in terms of unsatisfiability, one for the principles in terms of finite unsatisfiability, and the final one for a comparative study of both these types. The interconnections between these principles of explosion and the characterizations of these principles have been discussed in these subsections. In the final section, we discuss some directions for future research.

\section{Abstract model structures and related stuff\label{sec:amst}}

An \emph{abstract model structure} ($\amst$), introduced in \cite{RoyBasuChakraborty2025}, is a triple of the form $\mstr=(\mathbf{M},\models,\pow(\lang))$, where $\mathbf{M}\,(\ne\emptyset),\lang$ are sets, and $\models\,\subseteq\,\mathbf{M}\times \pow(\lang)$. One can think of $\lang$ as the set of \emph{sentences} or \emph{well-formed formulas (wffs)} of some language, $\mathbf{M}$ as a set of \emph{structures} (or \emph{models}) for $\lang$, and $\models$ as the \emph{satisfaction relation} between sets of sentences and models. Unless otherwise stated, $\lang$ is assumed to be infinite sets.

\begin{dfn}[Satisfiability, Finite Satisfiability, Compactness]
    Given an $\mathsf{amst}$ $\mstr=(\mathbf{M},\models,\pow(\lang))$, $m\in \mathbf{M}$, and a set $\Gamma\subseteq \lang$, we say that $m$ \emph{satisfies} $\Gamma$ \emph{in $\mstr$} if $m\models \Gamma$. A set $\Gamma\subseteq \lang$ is said to be \emph{satisfiable in $\mstr$} if there exists $m\in \mathbf{M}$ such that $m$ satisfies $\Gamma$, and \emph{finitely satisfiable in $\mstr$} if every finite subset of $\Gamma$ is satisfiable. An $\amst$ $\mstr=(\mathbf{M},\models,\pow(\lang))$ is said to be \emph{compact} if for all $\Gamma\subseteq \lang$, $\Gamma$ is satisfiable in $\mstr$ iff it is finitely satisfiable in $\mstr$.
\end{dfn}
We drop an explicit mention of $\mstr$, for discussions involving a single $\amst$.

Given an $\mathsf{amst}$ $\mstr=(\mathbf{M},\models,\pow(\lang))$, we define a map $\Mod: \pow(\lang)\to \pow(\mathbf{M})$ as follows. For all $\Gamma\subseteq \lang$,
\[
\Mod(\Gamma)=\{m\in \mathbf{M}\mid\,m\models\Gamma\}.
\]

\begin{dfn}[Normal $\amst$]\label{dfn:normamst}
    An $\mathsf{amst}$ $\mstr=(\mathbf{M},\models,\pow(\lang))$ is said to be \emph{normal} if,  for all $m\in\mathbf{M}$ and for all $\Gamma\subseteq \lang$, $m\models\Gamma$ iff $m\models\{\alpha\}$ for all $\alpha\in\Gamma$. 
\end{dfn}

The following result about normal $\amst$s was proved in \cite{RoyBasuChakraborty2025}.

\begin{thm}[{\cite[Theorem 3.10]{RoyBasuChakraborty2025}}]\label{thm:Prop_of_Mod}
    Suppose $\mstr=(\mathbf{M},\models,\pow(\lang))$ is a normal $\mathsf{amst}$. Then, the following statements hold.
\begin{enumerate}[label=(\roman*)]
    \item For all $\Gamma\subseteq\lang$, $\Mod(\Gamma)=\displaystyle\bigcap_{\alpha\in \Gamma}\Mod(\{\alpha\})$.
    \item For all $\Gamma\subseteq\Sigma\subseteq\lang$, $\Mod(\Sigma)\subseteq \Mod(\Gamma)$.
    \item For any family $(\Sigma_i)_{i\in I}$ of subsets of $\lang$,  
    \[
    \Mod\left(\displaystyle\bigcup_{i\in I}\Sigma_i\right)=\displaystyle\bigcap_{i\in I}\Mod(\Sigma_i)\quad\hbox{ and }\quad\displaystyle\bigcup_{i\in I}\Mod(\Sigma_i)\subseteq\Mod\left(\displaystyle\bigcap_{i\in I}\Sigma_i\right).
    \]
\end{enumerate}
\end{thm}

\begin{dfn}[Logical Structure induced by an $\amst$]\label{dfn:mlogstr}
    Suppose $\mstr=(\mathbf{M},\models,\pow(\lang))$ is an $\amst$. Then, the \emph{logical structure induced by $\mstr$}, denoted by $\mathcal{S}_{\mstr}=(\lang,\vdash_{\mstr})$, is such that $\vdash_{\mstr}$ is defined as follows. For all $\Gamma\cup\{\alpha\}\subseteq \lang$, 
    \[
    \Gamma\vdash_{\mstr}\alpha\hbox{ iff, for all } m\in \mathbf{M}, \hbox{ if }m\models\Gamma\hbox{ then }m\models\{\alpha\},\hbox{ i.e., }\Mod(\Gamma)\subseteq\Mod(\{\alpha\}).
    \]
    We drop the subscript $\mstr$ when there is no ambiguity about the $\amst$ under consideration. A logical structure induced by an $\amst$ is called a \emph{model-theoretic logical structure}.    
\end{dfn}

\begin{dfn}[Tarski-type Logical Structure]\label{def:Tarski}
 A logical structure $(\lang, \vdash)$ is said to be of \emph{Tarski-type} if $\vdash$ satisfies the following properties.
 \begin{enumerate}[label=(\alph*)]
     \item For all $\Gamma\subseteq \lang$ and  $\alpha\in \Gamma$, $\Gamma\vdash\alpha$. (Reflexivity)
     \item For all $\Gamma\cup\Sigma\cup\{\alpha\}\subseteq \lang$, $\Gamma\vdash\alpha$ and $\Gamma\subseteq \Sigma$ implies that $\Sigma\vdash\alpha$. (Monotonicity)
     \item For all $\Gamma\cup\Sigma\cup\{\alpha\}\subseteq \lang$, $\Gamma\vdash\alpha$ if $\Sigma\vdash\alpha$ and for all $\beta\in \Sigma$, $\Gamma\vdash\beta$. (Transitivity) 
 \end{enumerate}   
 In terms of the operator $C_\vdash$, the above conditions can be stated as follows.
 \begin{enumerate}[label=(\alph*)]
\item For all $\Gamma\subseteq \lang$, $\Gamma\subseteq C_{\vdash}(\Gamma)$. (Reflexivity)
\item For all $\Gamma,\Sigma\subseteq \lang$, if $\Gamma\subseteq \Sigma$ then $C_{\vdash}(\Gamma)\subseteq C_{\vdash}(\Sigma)$. (Monotonicity)
\item For all $\Gamma,\Sigma\subseteq \lang$, if $\Sigma\subseteq {C_{\vdash}}(\Gamma)$ then ${C_{\vdash}}(\Sigma)\subseteq {C_{\vdash}}(\Gamma)$. (Transitivity)
\end{enumerate}
\end{dfn}

\begin{thm}[{\cite[Theorem 3.8]{RoyBasuChakraborty2025}}]\label{thm:Tarski-type}
    The logical structure induced by a normal $\amst$ is of Tarski-type.
\end{thm}

\begin{dfn}[Finitary]
    A logical structure $(\lang,\vdash)$ is said to be \emph{finitary} if, for all $\Gamma\cup\{\alpha\}\subseteq\lang$, $\Gamma\vdash\alpha$ implies that there exists a finite $\Gamma^\prime\subseteq\Gamma$ such that $\Gamma^\prime\vdash\alpha$. In other words, for any $\alpha\in C_\vdash(\Gamma)$, there exists a finite $\Gamma^\prime\subseteq\Gamma$ such that $\alpha\in C_\vdash(\Gamma^\prime)$. 
\end{dfn}

Given an $\amst$ $\mstr=(\mathbf{M},\models,\pow(\lang))$, for each $\alpha\in\lang$, we define an $\amst$ $\mstr_\alpha=(\mathbf{M}_\alpha,\models_\alpha,\pow(\lang))$, where $\mathbf{M}_\alpha=\mathbf{M}\setminus\Mod(\{\alpha\})$ and $\models_\alpha\,\subseteq\mathbf{M}_\alpha\times\pow(\lang)$ is defined as follows. For all $m\in\mathbf{M}_\alpha$ and $\Gamma\subseteq\lang$, $m\models_\alpha\Gamma$ iff $m\models\Gamma$, i.e., $\models_\alpha\,=\,(\mathbf{M}_\alpha\times\pow(\lang))\,\cap\models$. Given an $\alpha\in\lang$, $\mathbf{M}_\alpha$ contains all those elements of $\mathbf{M}$ that `do not satisfy $\alpha$', i.e., $\mathbf{M}_\alpha=\{m\in\mathbf{M}\mid\,m\not\models\{\alpha\}\}$. In case of classical logic, this amounts to the set of all `models' of `not-$\alpha$.' Thus this construction can be seen as an attempt at generalizing the semantic handling of `negation.' We call $\mstr_\alpha$ the \emph{anti-model structure relative to $\alpha$}.

The following theorem characterizes finitary model-theoretic logical structures in terms of compactness using the above idea.

\begin{thm}[Characterization of Finitary-ness]\label{thm:char_finitary}
    Suppose $\mstr=(\mathbf{M},\models,\pow(\lang))$ is an $\amst$. Let, for any $\alpha\in\lang$, $\mstr_\alpha$ be as described above and $\mathcal{S}=(\lang,\vdash)$ be the logical structure induced by $\mstr$. Then, $\mathcal{S}$ is finitary iff, for all $\Gamma\cup\{\alpha\}\subseteq\lang$, $\Gamma$ is finitely satisfiable in $\mstr_\alpha$ implies that it is satisfiable in $\mstr_\alpha$.
\end{thm}

\begin{proof}
     Suppose $\mathcal{S}$ is finitary. If possible, let $\Gamma\cup\{\alpha\}\subseteq\lang$ such that $\Gamma$ is finitely satisfiable, but not satisfiable in $\mstr_\alpha$. So, for all finite $\Gamma^\prime\subseteq \Gamma$, there exists $n\in\mathbf{M}_\alpha$ such that $n\models_\alpha\Gamma^\prime$, i.e., $n\models\Gamma^\prime$, but $m\not\models_\alpha\Gamma$, i.e., $m\not\models\Gamma$, for all $m\in\mathbf{M}_\alpha$. Thus, for each finite $\Gamma^\prime\subseteq\Gamma$, $\Mod(\Gamma^\prime)\cap\mathbf{M}_\alpha\neq\emptyset$, while $\Mod(\Gamma)\cap\mathbf{M}_\alpha=\emptyset$, where $\Mod:\pow(\lang)\to\pow(\mathbf{M})$ as described before. Hence, for each finite $\Gamma^\prime\subseteq\Gamma$, $\Mod(\Gamma^\prime)\not\subseteq\Mod(\{\alpha\})$, i.e., $\Gamma^\prime\not\vdash\alpha$, while $\Mod(\Gamma)\subseteq\Mod(\{\alpha\})$, i.e., $\Gamma\vdash\alpha$. This contradicts the assumption that $\mathcal{S}$ is finitary. Hence, $\mstr_\alpha$ is compact for each $\alpha\in\lang$.

     Conversely, suppose $\mstr_\alpha$ is compact, for each $\alpha\in\lang$, but $\mathcal{S}$ is not finitary. Then, there exists $\Gamma\cup\{\alpha\}\subseteq \lang$ such that $\Gamma\vdash\alpha$ but $\Gamma^\prime\not\vdash\alpha$ for all finite $\Gamma^\prime\subseteq \Gamma$. So, for each finite $\Gamma^\prime\subseteq \Gamma$, there exists $m_{\Gamma^\prime}\in \mathbf{M}$ such that $m_{\Gamma^\prime}\models \Gamma^\prime$ but $m_{\Gamma^\prime}\not\models \{\alpha\}$. This implies that $\Gamma$ is finitely satisfiable in $\mstr_\alpha$. Now, as $\mstr_\alpha$ is compact for each $\alpha\in\lang$, $\Gamma$ is satisfiable in $\mstr_\alpha$. So, there exists $m_{\Gamma}\in \mathbf{M}_\alpha$ such that $m_{\Gamma}\models_\alpha \Gamma$. Since $m_\Gamma\in \mathbf{M}_\alpha$, $m_\Gamma\notin \Mod(\{\alpha\})$. Thus, $m_\Gamma\models \Gamma$ but $m_\Gamma\not\models\{\alpha\}$, which implies that $\Gamma\not\vdash\alpha$. This is a contradiction. Hence, $\mathcal{S}$ is finitary.
\end{proof}

Let $\mstr=(\mathbf{M},\models,\pow(\lang))$ be an $\amst$. Instead of constructing an anti-model structure relative to an $\alpha\in\lang$, we can also construct an anti-model structure relative to a set $\Lambda\subseteq\lang$ in the same way. Thus, for each $\Lambda\subseteq\lang$, $\mstr_\Lambda=(\mathbf{M},\models_\Lambda,\pow(\lang))$, where $\mathbf{M}_\Lambda=\mathbf{M}\setminus\Mod(\Lambda)$ and $\models_\Lambda\,=\,\mathbf{M}_\Lambda\cap\models$. This gives us the following characterization of finitary normal $\amst$s. 

\begin{thm}[Characterization of Finitary-ness for Normal $\amst$s]\label{thm:char_finnormal}
   Suppose $\mstr=(\mathbf{M},\models,\pow(\lang))$ is a normal $\amst$ and $\mathcal{S}=(\lang,\vdash)$ is the logical structure induced by $\mstr$. Then, $\mathcal{S}$ is finitary iff, for all finite $\Lambda\subseteq \lang$, $\mstr_\Lambda$ is compact.
\end{thm}

\begin{proof}
    Suppose $\mathcal{S}$ is finitary but there exists a finite $\Lambda\subseteq\lang$ such that $\mstr_\Lambda$ is not compact. Then, there exists $\Gamma\subseteq \lang$ such that $\Gamma$ is finitely satisfiable, but not satisfiable in $\mstr_\Lambda$. So, for all finite $\Gamma^\prime\subseteq \Gamma$, there exists $n\in\mathbf{M}_\Lambda$ such that $n\models_\Lambda\Gamma^\prime$, i.e., $n\models\Gamma^\prime$, but $m\not\models_\Lambda\Gamma$, i.e., $m\not\models\Gamma$, for all $m\in\mathbf{M}_\Lambda$. Then, by similar arguments as in the proof of Theorem \ref{thm:char_finitary}, we can conclude that for each finite $\Gamma^\prime\subseteq\Gamma$, $\Mod(\Gamma^\prime)\not\subseteq\Mod(\Lambda)$, but $\Mod(\Gamma)\subseteq\Mod(\Lambda)$. 
    
    Now, since $\mstr$ is normal, by Theorem \ref{thm:Prop_of_Mod}(i), $\Mod(\Lambda)=\displaystyle\bigcap_{\alpha\in \Lambda}\Mod(\{\alpha\})$. Hence, $\Mod(\Gamma)\subseteq \Mod(\{\alpha\})$ for all $\alpha\in \Lambda$. Thus, $\Gamma\vdash\alpha$ for all $\alpha\in \Lambda$. Now, as $\mathcal{S}$ is finitary, for each $\alpha\in \Lambda$, there exists a finite $\Gamma_\alpha\subseteq \Gamma$ such that $\Gamma_\alpha\vdash\alpha$. Let $\Gamma_0=\displaystyle\bigcup_{\alpha\in \Lambda}\Gamma_\alpha$. Then, $\Gamma_0$ is a finite subset of $\Gamma$ as each $\Gamma_\alpha$ is finite and $\Lambda$ is finite. Since $\mstr$ is normal, $\mathcal{S}$ is of Tarski-type, and hence, monotonic by Theorem \ref{thm:Tarski-type}. Thus, $\Gamma_0\vdash\alpha$ for all $\alpha\in \Lambda$. So, $\Mod(\Gamma_0)\subseteq \displaystyle\bigcap_{\alpha\in \Lambda}\Mod(\{\alpha\})=\Mod(\Lambda)$. This, however contradicts our previous conclusion that $\Mod(\Gamma^\prime)\not\subseteq\Mod(\Lambda)$ for all finite $\Gamma^\prime\subseteq\Gamma$. Hence, $\mstr_\Lambda$ is compact for each finite $\Lambda\subseteq\lang$.

    Conversely, suppose $\mstr_\Lambda$ is compact for each finite $\Lambda\subseteq\lang$. Then, in particular, $\mstr_\alpha$ is compact for each $\alpha\in\lang$. So, by Theorem \ref{thm:char_finitary}, $\mathcal{S}$ is finitary.
\end{proof}

The next result was proved in \cite{RoyBasuChakraborty2025}. We now obtain a different proof of it as a corollary of the above theorem.

\begin{cor}[{\cite[Corollary 3.19]{RoyBasuChakraborty2025}}]\label{cor:finunsat=>compact}
    Suppose $\mstr=(\mathbf{M},\models,\pow(\lang))$ is a normal $\amst$ such that the logical structure induced by it $\mathcal{S}=(\lang,\vdash)$ is finitary. If there exists a finite unsatisfiable set, then $\mstr$ is compact.  
\end{cor}

\begin{proof}
    Let $\Pi$ be a finite unsatisfiable set. Since $\mathcal{S}$ is finitary, by Theorem \ref{thm:char_finnormal}, $\mstr_\Lambda$ is compact for all finite $\Lambda\subseteq \lang$. Thus, in particular, $\mstr_\Pi$ is compact. Now, since $\Pi$ is not satisfiable, i.e., $\Mod(\Pi)=\emptyset$, $\mstr_\Pi=\mstr$. Hence, $\mstr$ is compact. 
\end{proof}

The notion of finitary-ness is non-semantic. Thus, the main importance of Theorem~\ref{thm:char_finitary} lies in the fact that it provides a semantic equivalent of a non-semantic concept. However, given a finitary logical structure, we might often like to investigate whether it is compact with respect to a particular semantics. In order to answer questions like this, instead of checking compactness of a family of $\amst$s (as suggested by Theorem~\ref{thm:char_finitary}), a direct approach is perhaps more helpful. In the final theorem of this section, we investigate this issue and connect the compactness of an $\amst$ with the finitary-ness of its induced logical structure via \emph{explosiveness}.

\begin{dfn}[Explosive Set]
    Suppose $\mathcal{S}=(\lang,\vdash)$ is a logical structure. A set $\Gamma\subseteq\lang$ is said to be $\mathcal{S}$-\emph{explosive} or $\mathcal{S}$-\emph{trivial} if $\Gamma\vdash\alpha$ for all $\alpha\in\lang$, i.e., $C_\vdash(\Gamma)=\lang$.
\end{dfn}

\begin{thm}[{\cite[Theorem 3.9]{RoyBasuChakraborty2025}}]\label{thm:unsat=>triv}
    Suppose $\mstr=(\mathbf{M},\models,\pow(\lang))$ is a normal $\amst$ and $\mathcal{S}=(\lang,\vdash)$ is the logical structure induced by $\mstr$. If $\Gamma\subseteq\lang$ is not satisfiable, then it is trivial. The converse holds if $\mstr$ is normal and $\lang$ is not satisfiable.
\end{thm}

\begin{thm}\label{thm:compact<=>fin}
    Suppose $\mstr=(\mathbf{M},\models,\pow(\lang))$ is a normal $\amst$ such that, for all $\alpha\in \lang$, there exists  $\Lambda_\alpha\subseteq \lang$ such that $\Mod(\Lambda_\alpha)=\mathbf{M}\setminus \Mod(\{\alpha\})$. If $\mstr$ is compact, then $\mathcal{S}=(\lang,\vdash)$, the logical structure induced by it, is finitary. The converse holds if for all $\alpha\in \lang$, $\Lambda_\alpha$ is finite.
\end{thm}

\begin{proof}
    Suppose $\mstr$ is compact. Let $\Gamma\cup\{\alpha\}\subseteq \lang$ such that $\Gamma\vdash\alpha$. Then $\Mod(\Gamma)\subseteq \Mod(\{\alpha\})$. If $\Gamma$ is not satisfiable, then by compactness, there exists a finite $\Gamma_0\subseteq\Gamma$ that is not satisfiable. By Theorem \ref{thm:unsat=>triv}, this implies that $\Gamma_0$ is trivial. Then, in particular, $\Gamma_0\vdash\alpha$. 
    
    Now, suppose $\Gamma$ is satisfiable and $\alpha\in \lang$ such that $\Gamma\vdash\alpha$. If possible, let $\Gamma^\prime\not\vdash\alpha$ for all finite $\Gamma^\prime\subseteq \Gamma$. This implies that $\Mod(\Gamma)\subseteq \Mod(\{\alpha\})$ but for all finite $\Gamma^\prime\subseteq \Gamma$, $\Mod(\Gamma^\prime)\not\subseteq\Mod(\{\alpha\})$. So, for each finite $\Gamma^\prime\subseteq\Gamma$, there exists $m_{\Gamma^\prime}\in \mathbf{M}$ such that $m_{\Gamma^\prime}\models \Gamma^\prime$, but $m_{\Gamma^\prime}\not\models\{\alpha\}$, i.e., $m_{\Gamma^\prime}\models \Lambda_\alpha$. Then, by normality of $\mstr$, $m_{\Gamma^\prime}$ satisfies every finite subset of $\Lambda_\alpha$. Again, by normality of $\mstr$, this implies that $m_{\Gamma^\prime}$ satisfies every finite subset of $\Gamma\cup\Lambda_\alpha$, i.e., $\Gamma\cup\Lambda_\alpha$ is finitely satisfiable. Then, as $\mstr$ is compact, $\Gamma\cup\Lambda_\alpha$ is satisfiable. Let $n\in \mathbf{M}$ be such that $n\models \Gamma\cup\Lambda_\alpha$. So, by normality of $\mstr$, $n\models \Gamma$ and $n\models \Lambda_\alpha$, i.e., $n\not\models \{\alpha\}$. This, however, contradicts our assumption that $\Gamma\vdash\alpha$. Thus, there must exist a finite $\Gamma^\prime\subseteq\Gamma$ such that $\Gamma^\prime\vdash\alpha$. Hence, $\mathcal{S}$ is finitary.

    Conversely, suppose $\mathcal{S}$ is finitary. Moreover, suppose that, for every $\alpha\in\lang$, there exists a finite $\Lambda_\alpha\subseteq\lang$ such that $\Mod(\Lambda_\alpha)=\mathbf{M}\setminus\Mod(\{\alpha\})$. This implies that $\Mod(\Lambda_\alpha)\cap\Mod(\{\alpha\})=\emptyset$. Since $\mstr$ is normal, $\Mod(\Lambda_\alpha)\cap\Mod(\{\alpha\})=\Mod(\Lambda_\alpha\cup\{\alpha\})$ by Theorem \ref{thm:Prop_of_Mod}(iii). Thus, $\Mod(\Lambda_\alpha\cup\{\alpha\})=\emptyset$, i.e., $\Lambda_\alpha\cup\{\alpha\}$ is not satisfiable. So, $\Lambda_\alpha\cup\{\alpha\}$ is a finite unsatisfiable set. Since $\mstr$ is normal and $\mathcal{S}$ is finitary, by Corollary \ref{cor:finunsat=>compact}, $\mstr$ is compact. 
\end{proof}

As an application of this result, we show the compactness of the paracomplete logic $\mbst$ that was introduced in \cite{BasuJain2025} to study the concept of probabilities on sets in the presence of undeterminedness.

\begin{exa}
    Suppose $\mbst=(\lang,\vdash_{\mbst})$ is the logic described in \cite{BasuJain2025} and $\mathcal{V}$ be the set of all $\mbst$-valuations. Let $\mstr_{\mbst}=(\mathcal{V},\models,\pow(\lang))$ be the $\amst$ where $\models\,\subseteq\pow(\lang)\times\lang$ is defined as follows. For any $v\in\mathcal{V}$ and $\Gamma\in\pow(\lang)$, $v\models\Gamma$ iff $v(\Gamma)=\{1\}$, i.e., $v(\gamma)=1$ for all $\gamma\in\Gamma$.  Clearly, for all $v\in\mathcal{V}$ and any $\Gamma\subseteq\lang$, $v\models\Gamma$ iff $v\models\{\gamma\}$ for all $\gamma\in\Gamma$. Thus, $\mstr_{\mbst}$ is normal.

    Let $(\lang,\vdash_{\mstr_{\mbst}})$ be the logical structure induced by $\mstr_{\mbst}$. Then, for any $\Gamma\cup\{\alpha\}\subseteq\lang$, $\Gamma\vdash_{\mstr_{\mbst}}\alpha$ iff, for all $v\in\mathcal{V}$, if $v\models\Gamma$ then $v\models\{\alpha\}$, i.e., $\Gamma\vdash_{\mstr_{\mbst}}\alpha$ iff, for all $v\in\mathcal{V}$, if $v(\Gamma)=\{1\}$ then $v(\{\alpha\})=1$. Thus, $\Gamma\vdash_{\mstr_{\mbst}}\alpha$ iff $\Gamma\vdash_{\mbst}\alpha$. So, the logical structure induced by $\mstr_{\mbst}$ is $\mbst$.

    Now, for any $\alpha\in\lang$ and any $v\in\mathcal{v}$,
    if $v(\alpha)=0$, then $v((\alpha\limp\fstar\alpha)\land(\alpha\limp\neg\fstar\alpha))=1$. Conversely, if $v((\alpha\limp\fstar\alpha)\land(\alpha\limp\neg\fstar\alpha))=1$, then $v(\alpha\limp\fstar\alpha)=1$ and $v(\alpha\limp\neg\fstar\alpha)=1$. So, if $v(\alpha)=1$, then this implies that $v(\fstar\alpha)=1=v(\neg\fstar\alpha)$, which is impossible. Thus, $v(\alpha)=0$. Hence, $v(\alpha)=0$ iff $v((\alpha\limp\fstar\alpha)\land(\alpha\limp\neg\fstar\alpha))=1$.

    Let $\Lambda_\alpha=\{(\alpha\limp\fstar\alpha)\land(\alpha\limp\neg\fstar\alpha)\}$. Then, by the above arguments,
    \[
    \begin{array}{lcl}
         \Mod(\Lambda_\alpha)&=&\{v\in\mathcal{V}\mid\,v\models\Lambda_\alpha\}\\
         &=&\{v\in\mathcal{V}\mid\,v((\alpha\limp\fstar\alpha)\land(\alpha\limp\neg\fstar\alpha))=1\}\\
         &=&\{v\in\mathcal{V}\mid\,v(\alpha)=0\}\\
         &=&\{v\in\mathcal{V}\mid\,v\not\models\{\alpha\}\}\\
         &=&\mathcal{V}\setminus\Mod(\{\alpha\})\\
    \end{array}
    \]
    Since $\Lambda_\alpha$ is finite, we can conclude that for each $\alpha\in\lang$, there exists a finite $\Lambda_\alpha\subseteq\lang$ such that $\Mod(\Lambda_\alpha)=\mathcal{V}\setminus\Mod(\{\alpha\})$. Hence, by Theorem \ref{thm:compact<=>fin}, $\mstr_{\mbst}$ is compact.
\end{exa}

\section{Semantic explosion principles\label{sec:sem}}
  We now define below some new notions of explosion for abstract model structures. These may be thought of as the `semantic' analogues of the explosion principles introduced in \cite{BasuRoy2024}. In these new principles, explosion has been replaced by unsatisfiability. As mentioned earlier, explosion from the semantic perspective is the result of unsatisfiability of a (contradictory) set of hypotheses. The following theorem gives a partial justification for this.

\begin{thm}\label{thm:explosive<=>unsat}
    Suppose $\mstr=(\mathbf{M},\models,\pow(\lang))$ is an $\amst$ and $\mathcal{S}=(\lang,\vdash)$ is the logical structure induced by $\mstr$. Then, for any $\Gamma\subseteq\lang$, if $\Gamma$ is not satisfiable, then it is $\mathcal{S}$-explosive. The converse holds in case there exists $\alpha\in\lang$ such that $\{\alpha\}$ is not satisfiable.
\end{thm}

\begin{proof}
    Suppose $\Gamma\subseteq\lang$ is not satisfiable. Let $\Gamma$ be not $\mathcal{S}$-explosive. Then, there exists $\gamma\in\lang$ such that $\Gamma\not\vdash\gamma$. So, there exists $m\in\mathbf{M}$ such that $m\models\Gamma$ but $m\not\models\{\gamma\}$. This, however, implies that $\Gamma$ is satisfiable, which is a contradiction. Hence, $\Gamma$ is $\mathcal{S}$-trivial.

    Conversely, suppose $\Gamma$ is $\mathcal{S}$-explosive. Let $\alpha\in\lang$ be such that $\{\alpha\}$ is not satisfiable, i.e., $\Mod(\{\alpha\})=\emptyset$. Since $\Gamma$ is $\mathcal{S}$-explosive, $\Gamma\vdash\alpha$. Then, $\Mod(\Gamma)\subseteq\Mod(\{\alpha\})$. So, $\Mod(\Gamma)=\emptyset$, i.e., $\Gamma$ is not satisfiable.
\end{proof} 
  
Due to the above arguments, it seems legitimate to refer to the principles introduced below as principles of explosion. Since compactness of $\amst$s is not assumed, in general, we get these principles in two variants: the $\mathsf{sat}$ (unsatisfiability) - variants and the $\mathsf{finsat}$ (finite unsatisfiability) - variants. These are separately taken up in the subsections below.

 \subsection{Semantic explosion principles - the \texorpdfstring{$\mathsf{sat}$}{sat} variants\label{subsec:sat}} 
\subsubsection{Definitions and interconnections}
\begin{dfn}[Semantic Explosion Principles: $\mathsf{sat}$ variants]\label{dfn:sat-explosion}
    Suppose $\mstr=(\mathbf{M},\models,\pow(\lang))$ is an $\amst$.
    \begin{enumerate}[label=(\roman*)]
        \item $\gecq$ holds in $\mstr$ if, for all $\alpha\in\lang$, there exists $\beta\in\lang$ such that $\{\alpha,\beta\}$ is not satisfiable.
        \item $\secq$ holds in $\mstr$ if, for all $\alpha\in\lang$, there exists $\Gamma\subseteq\lang$ such that $\Gamma\cup\{\alpha\}\subsetneq\lang$ and $\Gamma\cup\{\alpha\}$ is not satisfiable.
        \item $\specq$ holds in $\mstr$ if, for all $\Gamma\subsetneq\lang$, there exists $\alpha\in\lang$ such that $\Gamma\cup\{\alpha\}\subsetneq\lang$ and $\Gamma\cup\{\alpha\}$ is not satisfiable.
        \item $\pfecq$ holds in $\mstr$ if, for all $\Gamma\subsetneq\lang$, there exists $\Delta\subsetneq\lang$ such that $\Gamma\subseteq\Delta$ and $\Delta$ is not satisfiable.
    \end{enumerate}
    The letters $\mathsf{g},\mathsf{s}$ and the letter-pairs $\mathsf{sp},\mathsf{pf}$ stand for, respectively, `generalized,' `set-based,' `set-point,' and `point-free.' The postfix $\mathsf{sat}$ in the names of the above principles indicate that these are formulated in terms of (un)satisfiability.
\end{dfn}

\begin{rem}\label{rem:alt_secq}
    Suppose $\mstr=(\mathbf{M},\models,\pow(\lang))$ is an $\amst$. Suppose $\secq$ holds in $\mstr$ and $\alpha\in\lang$. Then, there exists $\Gamma\subseteq\lang$ such that $\alpha\in\Delta=\Gamma\cup\{\alpha\}\subsetneq\lang$ and $\Delta$ is not satisfiable.

    On the other hand, if, for every $\alpha\in\lang$, there exists $\Delta\subsetneq\lang$ such that $\alpha\in\Delta$ and $\Delta$ is not satisfiable, then $\Delta\cup\{\alpha\}=\Delta\subsetneq\lang$, which implies that $\secq$ holds. 

    Thus, an alternative formulation of $\secq$ is as follows. For all $\alpha\in\lang$, there exists $\Gamma\subsetneq\lang$ such that $\Gamma$ is not satisfiable.
\end{rem}

The next two theorems show that $\gecq,\secq,\specq,\pfecq$ are indeed the `semantic' analogues of gECQ, sECQ, spECQ, and pfECQ in \cite{BasuRoy2024}.

\begin{thm}\label{thm:amst->logstr}
    Suppose $\mstr=(\mathbf{M},\models,\pow(\lang))$ is an $\amst$ and $\mathcal{S}=(\lang,\vdash)$ is the logical structure induced by $\mstr$. 
    \begin{enumerate}[label=(\roman*)]
        \item If $\gecq$ holds in $\mstr$, then gECQ holds in $\mathcal{S}$.
        \item If $\secq$ holds in $\mstr$, then sECQ holds in $\mathcal{S}$.
        \item If $\specq$ holds in $\mstr$, then spECQ holds in $\mathcal{S}$.
        \item If $\pfecq$ holds in $\mstr$, then pfECQ holds in $\mathcal{S}$.
    \end{enumerate} 
\end{thm}

\begin{proof}
    \begin{enumerate}[label=(\roman*)]
        \item Suppose $\gecq$ holds in $\mstr$. Let $\alpha\in\lang$. Then, there exists $\beta\in\lang$ such that $\{\alpha,\beta\}$ is not satisfiable. Now, suppose there exists $\gamma\in\lang$ such that $\{\alpha,\beta\}\not\vdash\gamma$. So, there exists $m\in\mathbf{M}$ such that $m\models\{\alpha,\beta\}$ but $m\not\models\{\gamma\}$. However, this implies that $\{\alpha,\beta\}$ is satisfiable, which is a contradiction. Thus, $\{\alpha,\beta\}\vdash\gamma$ for all $\gamma\in\lang$. Hence, for every $\alpha\in\lang$, there exists $\beta\in\lang$ such that $\{\alpha,\beta\}$ is $\mathcal{S}$-explosive, i.e., gECQ holds in $\mathcal{S}$.

        \item Suppose $\secq$ holds in $\mstr$ and let $\alpha\in\lang$. Then, by Remark \ref{rem:alt_secq}, there exists $\Gamma\subsetneq\lang$ such that $\alpha\in\Gamma$ and $\Gamma$ is not satisfiable. Now, suppose there exists $\gamma\in\lang$ such that $\Gamma\not\vdash\gamma$. Then, there exists $m\in\mathbf{M}$ such that $m\models\Gamma$ but $m\not\models\{\gamma\}$. However, this implies that $\Gamma$ is satisfiable, which is a contradiction. Thus, $\Gamma\vdash\gamma$ for all $\gamma\in\lang$, i.e., $\Gamma$ is $\mathcal{S}$-explosive. Hence, sECQ holds in $\mathcal{S}$.
    \end{enumerate}
    Statements (iii) and (iv) can be established by similar arguments.
\end{proof}

\begin{thm}\label{thm:logstr->amst}
    Suppose $\mathcal{S}=(\lang,\vdash)$ is a logical structure and $\mstr=(\lang,\models,\pow(\lang))$ be an $\amst$, where $\models\,\subseteq\lang\times\pow(\lang)$ is defined as follows. For all $\Gamma\cup\{\alpha\}\subseteq\lang$, $\alpha\models\Gamma$ iff $\Gamma\not\vdash\alpha$.
    \begin{enumerate}[label=(\roman*)]
        \item gECQ holds in $\mathcal{S}$ iff $\gecq$ holds in $\mstr$.
        \item sECQ holds in $\mathcal{S}$ iff $\secq$ holds in $\mstr$.
        \item spECQ holds in $\mathcal{S}$ iff $\specq$ holds in $\mstr$.
        \item pfECQ holds in $\mathcal{S}$ iff $\pfecq$ holds in $\mstr$.
    \end{enumerate}
\end{thm}

\begin{proof}
    \begin{enumerate}[label=(\roman*)]
        \item Suppose gECQ holds in $\mathcal{S}$. Let $\alpha\in\lang$. Then, there exists $\beta\in\lang$ such that, for all $\gamma\in\lang$, $\{\alpha,\beta\}\vdash\gamma$. Now, suppose $\{\alpha,\beta\}$ is satisfiable. Then there exists $\delta\in\lang$ such that $\delta\models\{\alpha,\beta\}$. This implies that $\{\alpha,\beta\}\not\vdash\delta$. This is a contradiction. Thus, $\{\alpha,\beta\}$ is not satisfiable. Hence, $\gecq$ holds in $\mstr$.

        Conversely, let $\gecq$ holds in $\mstr$. Then, there exists $\beta\in\lang$ such that $\{\alpha,\beta\}$ is not satisfiable, i.e., for all $\gamma\in \lang$, $\gamma\not\models\{\alpha,\beta\}$, or equivalently, for all $\gamma\in \lang$, $\{\alpha,\beta\}\vdash\gamma$. We have thus shown that for all $\alpha\in \lang$, there exists $\beta\in \lang$ such that $\{\alpha,\beta\}$ is trivial. Hence, gECQ holds in $\mathcal{S}$.
        
        \item Suppose sECQ holds in $\mathcal{S}$. Let $\alpha\in\lang$. Then, there exists $\Gamma\subsetneq\lang$ such that $\alpha\in\Gamma$ and $\Gamma\vdash\gamma$ for all $\gamma\in\lang$. This implies that $\gamma\not\models\Gamma$ for all $\gamma\in\lang$, i.e., $\Gamma$ is not satisfiable. Hence, $\secq$ holds in $\mstr$, by Remark \ref{rem:alt_secq}. Converse holds by an argument similar to (i).
    \end{enumerate}
    Statements (iii) and (iv) can be established using similar arguments.
\end{proof}

We now investigate the interconnections between the explosion principles introduced in this section so far.

\begin{thm}Implications Between $\mathsf{sat}$-variants]{\label{thm:sat_imp}}
    Suppose $\mstr=(\mathbf{M},\models,\pow(\lang))$ is an $\amst$. Then, the following statements hold.
    \begin{enumerate}[label=(\roman*)]
        \item If $\specq$ holds in $\mstr$, then $\gecq$ holds in it.\label{thm:specq=>gecq}
        \item If $\specq$ holds in $\mstr$ then $\pfecq$ also holds in it.\label{thm:specq=>pfecq}        
        \item If $\gecq$ holds in $\mstr$, then $\secq$ holds in $\mstr$.\label{thm:gecq=>secq}
        \item If $\specq$ holds in $\mstr$ then $\secq$ also holds in it.\label{thm:specq=>secq}  
        \item If $\pfecq$ holds in $\mstr$, then $\secq$ holds in $\mstr$.\label{thm:pfecq=>secq}        
    \end{enumerate}
\end{thm}

\begin{proof}
    \begin{enumerate}[label=(\roman*)]
        \item Suppose $\specq$ holds in $\mstr$. Then, in particular for any $\alpha\in\lang$, $\{\alpha\}\subsetneq\lang$ as $\lang$ is infinite, and there exists $\beta\in\lang$ such that $\{\alpha\}\cup\{\beta\}\subsetneq\lang$ and is not satisfiable. Thus, $\gecq$ holds in $\mstr$.
        \item Suppose $\specq$ holds in $\mstr$ and $\Gamma\subsetneq\lang$. Then, there exist $\alpha\in\Gamma$ such that $\Gamma\cup\{\alpha\}\subsetneq\lang$ and $\Gamma\cup\{\alpha\}$ is not satisfiable. Since $\Gamma\subseteq\Gamma\cup\{\alpha\}$, this implies that $\pfecq$ holds in $\mstr$.
        \item Suppose $\gecq$ holds in $\mstr$ and let $\alpha\in\lang$. Then, there exists $\beta\in\lang$ such that $\{\alpha,\beta\}$ is not satisfiable. Now, since $\lang$ is infinite, $\{\alpha,\beta\}\subsetneq\lang$. Since, $\{\alpha,\beta\}=\{\alpha,\beta\}\cup\{\alpha\}$ is not satisfiable, $\secq$ holds in $\mstr$.
        \item Since $\specq$ implies $\gecq$ by (i) and $\gecq$ implies $\secq$ by (iii), $\specq$ implies $\secq$.
        \item Suppose $\pfecq$ holds in $\mstr$ and $\alpha\in\lang$. Since $\lang$ is infinite, $\{\alpha\}\subsetneq\lang$. So, by $\pfecq$, there exists $\Delta\subsetneq\lang$ such that $\{\alpha\}\subseteq\Delta$ and $\Delta$ is not satisfiable. Then, $\alpha\in\Delta$ and hence, $\Delta\cup\{\alpha\}=\Delta\subsetneq\lang$ and is not satisfiable. Thus, $\secq$ holds in $\mstr$.
    \end{enumerate}
\end{proof}

The interconnections between the principles discussed so far are detailed in Figure \ref{fig:sat-exp}. The intended interpretation of the figure is as follows: given an $\amst$ if it satisfies the principle of explosion at the tail of an arrow, it also satisfies the principle of explosion at the tip of the same arrow.

\begin{figure}[H]
    \centering
\[\begin{tikzcd}
	{\mathsf{gECQ\text{-}sat}} && {\mathsf{sECQ\text{-}sat}} \\
	\\
	{\mathsf{spECQ\text{-}sat}} && {\mathsf{pfECQ\text{-}sat}}
	\arrow[from=1-1, to=1-3]
	\arrow[from=3-1, to=1-1]
	\arrow[from=3-1, to=3-3]
	\arrow[from=3-3, to=1-3]
\end{tikzcd}\]
    \caption{Semantic explosion principles - the $\mathsf{sat}$ variants}
    \label{fig:sat-exp}
\end{figure}
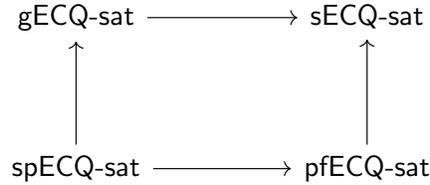

The rest of this subsection is aimed at showing that no other implication holds between these principles of explosion. To begin with, we first show that $\secq$ does not imply $\gecq$.

\begin{exa}[$\secq\centernot\implies\gecq$]\label{exm:secq≠>gecq}
    Let $\mstr=(\mathbf{M},\models,\pow(\lang))$ be an $\amst$, where $\models\,\mathbf{M}\times\pow(\lang)$ is defined as follows. For all $m\in\mathbf{M}$ and $\Gamma\subseteq\lang$, $m\models\Gamma$ iff $\Gamma$ is finite.

    Then, as $\lang$ is infinite, for any $\alpha\in\lang$, there exists $\beta\in\lang\setminus\{\alpha\}$. Moreover, $(\lang\setminus\{\alpha,\beta\})\cup\{\alpha\}=\lang\setminus\{\beta\}\subsetneq\lang$ and is not satisfiable as it is infinite. Thus, $\secq$ holds in $\mstr$.

    However, for any $\alpha\in\lang$, the set $\{\alpha,\beta\}$ for any $\beta\in\lang$ is satisfiable since it is finite. Thus, $\gecq$ fails in $\mstr$.
\end{exa}

\begin{rem}
\begin{enumerate}[label=(\roman*)]
    \item Note that $\secq$ does not imply $\specq$. Otherwise, since $\specq$ implies $\gecq$, it would follow that $\secq$ implies $\gecq$ $-$ contradicting the above example. 
    \item Let $\mstr$ be the same $\amst$ as in the above example. Let $\Gamma\subsetneq\lang$. If $\Gamma$ is finite, then as $\lang$ is infinite, there exists $\beta\in\lang\setminus\Gamma$. So, $\Gamma\subseteq\lang\setminus\{\beta\}\subsetneq\lang$. Clearly, $\lang\setminus\{\beta\}$ is infinite and hence not satisfiable. On the other hand, if $\Gamma$ is infinite, then it is already not satisfiable. Thus, $\pfecq$ holds in $\mstr$. Thus, Example \ref{exm:secq≠>gecq} also shows that $\pfecq$ does not imply $\gecq$.
    \item Finally, $\pfecq$ does not imply $\specq$. Otherwise, as $\specq$ implies $\gecq$ by Theorem \ref{thm:sat_imp}(i), it would follow that $\pfecq$ implies $\gecq$ $-$ contradicting (ii).
\end{enumerate}    
\end{rem}

We now show that the converse of Theorem~\ref{thm:sat_imp}(i) is not true. The same example also proves that neither $\gecq$ nor $\secq$ imply either $\specq$ or $\pfecq$.

\begin{exa}[$\gecq/\secq\centernot\implies\specq/\pfecq$]\label{exm:gecq/secq≠>pfecq/specq}
    Let $\mstr=(\NN,\models,\pow(\NN))$ be an $\amst$, where for any $\Gamma\cup\{m\}\subseteq\NN$, $m\models\Gamma$ iff $\Gamma\neq\{n,n+1\}$ for some $n\in\NN$. Now, for each $n\in\NN$, $\{n,n+1\}$ is not satisfiable. Hence, $\gecq$ holds in $\mstr$.

    Moreover, for each $n\in\NN$, the set $\{n,n+1\}$ is not satisfiable. This implies that $\gecq$ holds in $\mstr$. So, by Theorem \ref{thm:sat_imp}(iii), $\secq$ also holds in $\mstr$.

    However, for any set $\Gamma\subseteq\lang$ with $\lvert\Gamma\rvert\ge3$, there does not exist a $\Delta\supseteq\Gamma$ such that $\Delta$ is not satisfiable. Hence, $\pfecq$ does not hold in $\mstr$. 
    
    In particular, there does not exist a $k\in\NN$ such that $\Gamma\cup\{k\}$ is not satisfiable. Thus, $\specq$ fails in $\mstr$.
\end{exa}
\subsubsection{Characterization theorems}
We now turn to characterization results for the four principles of explosion in this subsection.

\begin{thm}[Characterization of $\secq$ for Normal $\amst$s]\label{thm:char_secq_normal}
    Suppose $\mstr=(\mathbf{M},\models,\pow(\lang))$ is a normal $\amst$. $\secq$ holds in $\mstr$ iff for each $\alpha\in\lang$, there exists $\beta\in\lang\setminus\{\alpha\}$ such that $\lang\setminus\{\beta\}$ is not satisfiable.
\end{thm}

\begin{proof}
    Suppose $\secq$ holds in $\mstr$ and $\alpha\in\lang$. Then, there exists $\Gamma\subseteq\lang$ such that $\Gamma\cup\{\alpha\}\subsetneq\lang$ and $\Gamma\cup\{\alpha\}$ is not satisfiable. Since $\Gamma\cup\{\alpha\}\subsetneq\lang$, there exists $\beta\in\lang\setminus(\Gamma\cup\{\alpha\})$. Then, $\Gamma\cup\{\alpha\}\subseteq\lang\setminus\{\beta\}$. Now, as $\Gamma\cup\{\alpha\}$ is not satisfiable and $\mstr$ is normal, $\lang\setminus\{\beta\}$ is not satisfiable.

    Conversely, suppose for each $\alpha\in\lang$, there exists $\beta\in\lang\setminus\{\alpha\}$ such that $\lang\setminus\{\beta\}$ is not satisfiable. Then, $\alpha\in\lang\setminus\{\beta\}$ and hence, $(\lang\setminus\{\beta\})\cup\{\alpha\}=\lang\setminus\{\beta\}\subsetneq\lang$. Moreover, $\lang\setminus\{\beta\}$ is not satisfiable. Thus, $\secq$ holds in $\mstr$.
 \end{proof}

\begin{cor}[Characterization of $\gecq$ for Normal $\amst$s]\label{cor:char_gecq_normal}
    Suppose $\mstr=(\mathbf{M},\models,\pow(\lang))$ is a normal $\amst$. $\gecq$ holds in $\mstr$ iff for each $\alpha\in\lang$, there exists $\beta\in\lang$ such that $\lang\setminus\{\beta\}$ is not satisfiable.
\end{cor}

\begin{proof}
    Suppose $\gecq$ holds in $\mstr$. Then, by Theorem \ref{thm:sat_imp}(iii), $\secq$ holds in $\mstr$. So, by Theorem \ref{thm:char_secq_normal}, for each $\alpha\in\lang$, there exists $\beta\in\lang\setminus\{\alpha\}\subseteq\lang$ such that $\lang\setminus\{\beta\}$ is not satisfiable.

    Conversely, suppose for each $\alpha\in\lang$, there exists $\beta\in\lang$ such that $\lang\setminus\{\beta\}$ is satisfiable. Since, $\mstr$ is normal, this implies that $\{\alpha,\beta\}$ is not satisfiable. Hence, $\gecq$ holds in $\mstr$.
\end{proof}

To obtain a characterization for $\secq$ (without the assumption of normality), we introduce the concept of \emph{relativized satisfiability} below. Besides this, the notions of filters and principal ultrafilters are also required. These are also explained below.

\begin{dfn}[Relativized Satisfiability]\label{dfn:rel_sat}
    Suppose $\mstr=(\mathbf{M},\models,\pow(\lang))$ is an $\amst$ and  $\mathcal{K}\subseteq\pow(\lang)$. Then, a set $\Sigma\subseteq\lang$ is said to be \emph{satisfiable relative to $\mathcal{K}$} or \emph{$\mathcal{K}$-satisfiable} if, for every $\Delta\subseteq\Sigma$ such that $\Delta\in\mathcal{K}$, $\Delta$ is satisfiable.
\end{dfn}

\begin{rem}
    Suppose $\mstr=(\mathbf{M},\models,\pow(\lang))$ is an $\amst$. The following statements are easy to verify.
    \begin{enumerate}[label=(\roman*)]
        \item Any satisfiable $\Gamma\subseteq\lang$ is satisfiable relative to $\{\Gamma\}$.
        \item Any finitely satisfiable $\Gamma$ is satisfiable relative to $\mathcal{K}=\{\Delta\subseteq\lang\mid\,\Delta\hbox{ is finite}\}$.
        \item If $\mstr$ is normal, then any satisfiable $\Gamma\subseteq\lang$ is satisfiable relative to $\pow(\lang)$.
    \end{enumerate}
\end{rem}

\begin{dfn}[Filters, Ultrafilters, Principal Ultrafilters]
Given a set $I$, a \emph{filter on $I$} is a set $\mathcal{F}\subseteq\pow(I)$ satisfying the following properties. 
\begin{enumerate}[label=(\alph*)]
    \item $X\in \mathcal{F}$ and $Y\in \mathcal{F}$ implies that $X\cap Y\in \mathcal{F}$.
    \item $X\subseteq Y$ and $X\in \mathcal{F}$ implies that $Y\in \mathcal{F}$.
\end{enumerate}
A filter $\mathcal{F}$ on $I$ is said to be a \emph{proper filter} if $\emptyset\notin \mathcal{F}$. 

An \emph{ultrafilter on $I$} is a maximal proper filter on $I$. In other words, an ultrafilter on $I$ is a proper filter $\mathcal{U}$ on $I$ such that for every filter $\mathcal{F}\supsetneq \mathcal{U}$, $\emptyset\in \mathcal{F}$.

Given an $s\in I$, the set $\mathcal{U}_s^I=\{X\subseteq I\mid\,s\in X\}$ is an ultrafilter on $I$ and is called the \emph{principal ultrafilter generated by $s$}. An ultrafilter $\mathcal{U}$ is called a \emph{principal filter} if $\mathcal{U}=\mathcal{U}_s^I$ for some $s\in I$.
\end{dfn}

\begin{thm}[Characterization of $\secq$]\label{thm:char_secq}
    Suppose $\mstr=(\mathbf{M},\models,\pow(\lang))$ is an $\amst$. Then, the following statements are equivalent.
    \begin{enumerate}[label=(\roman*)]
        \item $\secq$ holds in $\mstr$.
        \item For all $\alpha\in\lang$, there exists $\beta\in\lang\setminus\{\alpha\}$, such that $\lang\setminus\{\beta\}$ is not satisfiable relative to $\mathcal{U}_\alpha^{\lang\setminus\{\beta\}}$.
        \item For all $\alpha\in \mathscr{L}$, if $\mathscr{L}\setminus\{\alpha\}$ is satisfiable, then for all $\beta\in \mathscr{L}\setminus\{\alpha\}$, there exists $\Delta\subsetneq \mathscr{L}$ such that $\alpha\in \Delta$ and $\Delta$ is not satisfiable.
    \end{enumerate} 
\end{thm}

\begin{proof}
    \underline{(i) $\implies$ (ii)}: Suppose $\secq$ holds in $\mstr$ and $\alpha\in\lang$. Then, by Remark \ref{rem:alt_secq}, there exists $\Sigma\subsetneq\lang$ such that $\alpha\in\Sigma$ and $\Sigma$ is not satisfiable. Since $\alpha\in\Sigma\subsetneq\lang$, there exists $\beta\in\lang\setminus\{\alpha\}$ such that $\Sigma\subseteq\lang\setminus\{\beta\}$. Thus, $\Sigma\in\mathcal{U}_\alpha^{\lang\setminus\{\beta\}}$. Since $\Sigma$ is not satisfiable, this implies that $\lang\setminus\{\beta\}$ is not satisfiable relative to $\mathcal{U}_\alpha^{\lang\setminus\{\beta\}}$.

    \underline{(ii) $\implies$ (i)}: Suppose statement (ii) holds, and let $\alpha\in\lang$. Then, there exists $\beta\in\lang\setminus\{\alpha\}$ such that $\lang\setminus\{\beta\}$ is not satisfiable relative to $\mathcal{U}_\alpha^{\lang\setminus\{\beta\}}$. So, there exists $\Sigma\subseteq\lang\setminus\{\beta\}$ such that $\Sigma\in\mathcal{U}_\alpha^{\lang\setminus\{\beta\}}$ and $\Sigma$ is not satisfiable. Since $\Sigma\subseteq\lang\setminus\{\beta\}$, $\Sigma\subsetneq\lang$, and as $\Sigma\in\mathcal{U}_\alpha^{\lang\setminus\{\beta\}}$, $\alpha\in\Sigma$. Thus, for any $\alpha\in\lang$, there exists $\Sigma\subsetneq\lang$ such that $\alpha\in\Sigma$ and $\Sigma$ is not satisfiable. Hence, by Remark \ref{rem:alt_secq}, $\secq$ holds in $\mstr$.

     \underline{(i) $\implies$ (iii)}: Suppose $\secq$ holds in $\mstr$. Then, the statement (iii) immediately follows from the definition of $\secq$.

     \underline{(iii) $\implies$ (i)}: Suppose statement (iii) holds and let $\alpha\in\lang$. Suppose further that there exist $\beta,\gamma\in\lang$, with $\beta\ne\gamma$, such that $\lang\setminus\{\beta\}$ and $\lang\setminus\{\gamma\}$ are not satisfiable. If $\alpha\ne\beta$, then $\alpha\in \lang\setminus\{\beta\}$; otherwise, $\alpha\in \lang\setminus\{\gamma\}$. In either case, $\alpha$ is contained in a proper subset of $\lang$ that is not satisfiable. Thus, $\secq$ holds.
        
    On the other hand, if it is not the case that there are two distinct $\beta,\gamma\in\lang$ such that $\lang\setminus\{\beta\}$ and $\lang\setminus\{\gamma\}$ are not satisfiable, then there is at most one unsatisfiable set of the form $\lang\setminus\{\beta\}$. Let $\gamma\in\lang$ such that $\gamma\notin \{\alpha,\beta\}$ (such a $\gamma$ exists since $\lang$ is infinite). Then, as $\gamma\ne\beta$, $\lang\setminus\{\gamma\}$ is satisfiable. So, since $\alpha\in \lang\setminus\{\gamma\}$, by statement (iii), there exists $\Delta\subsetneq \lang$ such that $\alpha\in \Delta$ and $\Delta$ is not satisfiable. Thus, $\secq$ holds in this case as well. 
\end{proof}

\begin{thm}[Characterization of $\pfecq$]\label{thm:char_pfecq}
    Suppose $\mstr=(\mathbf{M},\models,\pow(\lang))$ is an $\amst$. $\pfecq$ holds in $\mstr$ iff, for all $\alpha\in\lang$, $\lang\setminus\{\alpha\}$ is not satisfiable.
\end{thm}

\begin{proof}
    Suppose $\pfecq$ holds in $\mstr$ and $\alpha\in\lang$. Then, as $\lang\setminus\{\alpha\}\subsetneq\lang$, there exists $\Delta\subsetneq\lang$ such that $\lang\setminus\{\alpha\}\subseteq\Delta$ and $\Delta$ is not satisfiable. Clearly, $\alpha\notin\Delta$, since otherwise, $\Delta=\lang$, a contradiction. Thus, $\Delta=\lang\setminus\{\alpha\}$. Hence, $\lang\setminus\{\alpha\}$ is not satisfiable.

    Conversely, suppose $\lang\setminus\{\alpha\}$ is not satisfiable for all $\alpha\in\lang$. Moreover, suppose $\pfecq$ fails in $\mstr$. Then, there exists $\Gamma\subsetneq\lang$ such that, for all $\Delta\subsetneq\lang$, if $\Gamma\subseteq\Delta$, then $\Delta$ is satisfiable. Now, as $\Gamma\subsetneq\lang$, there exists $\alpha\in\lang\setminus\Gamma$. Then, $\Gamma\subseteq\lang\setminus\{\alpha\}\subsetneq\lang$. So, $\lang\setminus\{\alpha\}$ is satisfiable, which contradicts our assumption. Hence, $\pfecq$ holds in $\mstr$. 
\end{proof}

Finally, the following series of results lead to a characterization of $\specq$.

\begin{thm}\label{thm:specq<=>pfecq}
    Suppose $\mstr=(\mathbf{M},\models,\pow(\lang))$ is an $\amst$. If $\specq$ holds in $\mstr$, then there exists $\varphi\in\lang$ such that $\{\varphi\}$ is not satisfiable, and $\pfecq$ holds in $\mstr$. Moreover, if for all $\Gamma\subsetneq\lang$, $\Gamma$ is finitely satisfiable whenever it is satisfiable, then the converse also holds.
\end{thm}

\begin{proof}
    Suppose $\specq$ holds in $\mstr$. Then, by Theorem \ref{thm:sat_imp}(ii), $\pfecq$ holds in $\mstr$. Now, as $\emptyset\subsetneq\lang$, by $\specq$, there exists $\varphi\in\lang$ such that $\emptyset\cup\{\varphi\}=\{\varphi\}\subsetneq\lang$ and $\{\varphi\}$ is not satisfiable.

    Conversely, suppose there exists $\varphi\in\lang$ such that $\{\varphi\}$ is not satisfiable and $\pfecq$ holds in $\mstr$. Moreover, we assume that for all $\Gamma\subsetneq\lang$, if $\Gamma$ is satisfiable, then it is finitely satisfiable. Suppose, if possible, $\specq$ does not hold in $\mstr$. So, there exists $\Delta\subsetneq\lang$ such that, for all $\alpha\in\lang$, if $\Delta\cup\{\alpha\}\subsetneq\lang$, then $\Delta\cup\{\alpha\}$ is satisfiable. We claim that this implies $\Delta$ is satisfiable.

    \textbf{Case 1:} $\Delta=\emptyset$

    In this case, since for any $\alpha\in\lang$, $\Delta\cup\{\alpha\}=\{\alpha\}\subsetneq\lang$, and hence, is satisfiable. So, by the assumed condition, $\{\alpha\}$ is finitely satisfiable. Thus, $\Delta=\emptyset$ being a finite subset of $\{\alpha\}$ is satisfiable.

    \textbf{Case 2:} $\Delta\neq\emptyset$

    Let $\alpha\in\Delta$. Then, $\Delta\cup\{\alpha\}=\Delta\subsetneq\lang$ is satisfiable. 

    Hence, in all cases, $\Delta$ is satisfiable. Now, since $\pfecq$ holds in $\mstr$, by Theorem \ref{thm:char_pfecq}, $\lang\setminus\{\alpha\}$ is not satisfiable for all $\alpha\in\lang$. Thus, $\Delta\neq\lang\setminus\{\alpha\}$ for all $\alpha\in\lang$. So, for any $\alpha\in\lang$, $\Delta\cup\{\alpha\}\subsetneq\lang$, and thus, $\Delta\cup\{\alpha\}$ is satisfiable, by the failure of $\specq$. Then, since by assumption, satisfiability implies finite satisfiability, $\{\alpha\}\subseteq\Delta\cup\{\alpha\}$ is satisfiable. This contradicts the assumption that there is no $\varphi\in\lang$ such that $\{\varphi\}$ is satisfiable. Hence, $\specq$ must hold in $\mstr$.
\end{proof}

\begin{cor}[Characterization of $\specq$ for Normal $\amst$s]\label{cor:char_specq_normal}
    Suppose $\mstr=(\mathbf{M},\models,\pow(\lang))$ is a normal $\amst$. $\specq$ holds in $\mstr$, iff there exists $\varphi\in\lang$ such that $\{\varphi\}$ is not satisfiable and $\pfecq$ holds in $\mstr$.
\end{cor}

\begin{proof}
    Since $\mstr$ is normal, for any $\Gamma\subseteq\lang$, if $\Gamma$ is satisfiable, then every subset of $\Gamma$, and hence, every finite subset of $\Gamma$ is satisfiable. Hence, the result follows from the above theorem.
\end{proof}

\begin{thm}[Characterization of $\specq$]\label{thm:char_specq}
    Suppose $\mstr=(\mathbf{M},\models,\pow(\lang))$ is an $\amst$. Then the following statements are equivalent.
    \begin{enumerate}[label=(\roman*)]
        \item $\specq$ holds in $\mstr$.
        \item There exists $\varphi\in\lang$ such that $\{\varphi\}$ is not satisfiable, and for any satisfiable $\Gamma\subsetneq\lang$, there exists $\gamma\in\lang$ such that $\Gamma\cup\{\gamma\}\subsetneq\lang$ and $\Gamma\cup\{\gamma\}$ is not satisfiable.
        \item There exists $\varphi\in\lang$ such that $\{\varphi\}$ is not satisfiable, $\pfecq$ holds in $\mstr$, and for all $\Gamma\subsetneq\lang$, for all $\alpha\in\lang$, there exists $\beta\in\lang$ such that $\Gamma\cup\{\beta\}\vdash\alpha$, where $(\lang,\vdash)=\mathcal{S}$ is the logical structure induced by $\mstr$.
    \end{enumerate}
\end{thm}

\begin{proof}
    \underline{(i) $\implies$ (ii)}: Suppose $\specq$ holds in $\mstr$. Then, by Theorem \ref{thm:specq<=>pfecq}, there exists $\varphi\in\lang$ such that $\{\varphi\}$ is not satisfiable. The rest of the statement (ii) follows immediately from the definition of $\specq$.

    \underline{(ii) $\implies$ (i)}: Suppose the statement (ii) holds. So, for any satisfiable $\Gamma\subsetneq\lang$, there exists $\gamma\in\lang$ such that $\Gamma\cup\{\gamma\}\subsetneq\lang$ and $\Gamma\cup\{\gamma\}$ is not satisfiable. Thus, to establish that $\specq$ holds in $\mstr$, it suffices to show that the same holds for any $\Gamma\subsetneq\lang$ that is not satisfiable, i.e., there exists $\gamma\in\lang$ such that $\Gamma\cup\{\gamma\}\subsetneq\lang$ and $\Gamma\cup\{\gamma\}$ is not satisfiable. Let $\Gamma\subsetneq\lang$ that is not satisfiable.

    \textbf{Case 1:} $\Gamma=\emptyset$

    By assumption, there exists $\varphi\in\lang$ such that $\{\varphi\}$ is not satisfiable. Now, $\Gamma\cup\{\varphi\}=\{\varphi\}\subsetneq\lang$, as $\lang$ is infinite, and $\Gamma\cup\{\varphi\}$ is not satisfiable.

    \textbf{Case 2:} $\Gamma\neq\emptyset$

    Let $\gamma\in\Gamma$. Then, $\Gamma\cup\{\gamma\}=\Gamma\subsetneq\lang$ and $\Gamma\cup\{\gamma\}$ is not satisfiable.

    Thus, for any $\Gamma\subsetneq\lang$, there exists $\gamma\in\lang$ such that $\Gamma\cup\{\gamma\}\subsetneq\lang$ and $\Gamma\cup\{\gamma\}$ is not satisfiable, i.e., $\specq$ holds in $\mstr$.

    \underline{(i) $\implies$ (iii)}: Suppose $\specq$ holds in $\mstr$. Then, by Theorem \ref{thm:specq<=>pfecq}, there exists $\varphi\in\lang$ such that $\{\varphi\}$ is not satisfiable and $\pfecq$ holds in $\mstr$. To show the remaining portion of the statement (iii), let $\Gamma\subsetneq\lang$ and $\alpha\in\lang$. Now, since $\specq$ holds in $\mstr$, there exists $\beta\in\lang$ such that $\Gamma\cup\{\beta\}\subsetneq\lang$ and $\Gamma\cup\{\beta\}$ is not satisfiable. Then, by Theorem \ref{thm:unsat=>triv}, $\Gamma\cup\{\beta\}$ is $\mathcal{S}$-explosive. Hence, $\Gamma\cup\{\beta\}\vdash\alpha$.

    \underline{(iii) $\implies$ (i)}: Suppose the statement (iii) holds but $\specq$ does not hold. Then, there exists $\Gamma\subsetneq\lang$ such that for all $\alpha\in\lang$, if $\Gamma\cup\{\alpha\}\subsetneq\lang$, then $\Gamma\cup\{\alpha\}$ is satisfiable. So, for all $\gamma\in\Gamma$, $\Gamma\cup\{\gamma\}=\Gamma\subsetneq\lang$, and hence, $\Gamma$ is satisfiable. Now, as $\pfecq$ holds in $\mstr$, by Theorem \ref{thm:char_pfecq}, $\lang\setminus\{\alpha\}$ is not satisfiable for all $\alpha\in\lang$. Thus, $\Gamma\neq\lang\setminus\{\alpha\}$ for all $\alpha\in\lang$. In other words, $\Gamma\cup\{\gamma\}\subsetneq\lang$ for all $\gamma\in\lang$.

    By assumption, there exists $\varphi\in\lang$ such that $\{\varphi\}$ is not satisfiable. Also, by our assumption, there exists $\beta\in\lang$ such that $\Gamma\cup\{\beta\}\vdash\varphi$. This implies that $\Gamma\cup\{\beta\}$ is not satisfiable since $\{\varphi\}$ is not satisfiable. Moreover, $\Gamma\cup\{\beta\}\subsetneq\lang$. This contradicts our assumption about $\Gamma$. Hence, $\specq$ holds in $\mstr$.
\end{proof}

\subsection{Semantic explosion principles - the \texorpdfstring{$\mathsf{finsat}$}{finsat} variants\label{subsec:finsat}}
\subsubsection{Definitions and interconnections}
We now obtain the following variants of explosion principles discussed in the previous subsection, in Definition \ref{dfn:sat-explosion}, by changing satisfiability to finite satisfiability.

\begin{dfn}[Semantic Explosion Principles: $\mathsf{finsat}$ variants]\label{dfn:finsat-explosion}
    Suppose $\mstr=(\mathbf{M},\models,\pow(\lang))$ is an amst.
    \begin{enumerate}[label=(\roman*)]
        \item $\fgecq$ holds in $\mstr$ if, for all $\alpha\in\lang$, there exists $\beta\in\lang$ such that $\{\alpha,\beta\}$ is not finitely satisfiable.
        \item $\fsecq$ holds in $\mstr$ if, for all $\alpha\in\lang$, there exists $\Gamma\subseteq\lang$ such that $\Gamma\cup\{\alpha\}\subsetneq\lang$ and $\Gamma\cup\{\alpha\}$ is not finitely satisfiable.
        \item $\fspecq$ holds in $\mstr$ if, for all $\Gamma\subsetneq\lang$, there exists $\alpha\in\lang$ such that $\Gamma\cup\{\alpha\}\subsetneq\lang$ and $\Gamma\cup\{\alpha\}$ is not finitely satisfiable.
        \item $\fpfecq$ holds in $\mstr$ if, for all $\Gamma\subsetneq\lang$, there exists $\Delta\subsetneq\lang$ such that $\Gamma\subseteq\Delta$ and $\Delta$ is not finitely satisfiable.
    \end{enumerate}
    The postfix $\mathsf{finsat}$ in the names of the above principles indicate that these are formulated in terms of finite (un)satisfiability.
\end{dfn}

\begin{rem}
    Clearly, if the $\amst$ under consideration is compact, then the -$\mathsf{finsat}$ variants are equivalent to their corresponding -$\mathsf{sat}$ counterparts. Thus, in the presence of compactness the implications that exist between the -$\mathsf{sat}$ principles will hold between the corresponding -$\mathsf{finsat}$ principles.   
\end{rem}

Using arguments similar to the one in Theorem \ref{thm:sat_imp}, one can easily show the following theorem.

\begin{thm}[Implications Between $\mathsf{finsat}$-variants]{\label{thm:finsat_imp}}
    Suppose $\mstr=(\mathbf{M},\models,\pow(\lang))$ is an $\amst$. Then, the following statements hold.
    \begin{enumerate}[label=(\roman*)]
        \item If $\fspecq$ holds in $\mstr$, then $\fgecq$ holds in it.\label{thm:fspecq=>fgecq}
        \item If $\fspecq$ holds in $\mstr$ then $\fpfecq$ also holds in it.\label{thm:fspecq=>fpfecq}        
        \item If $\fgecq$ holds in $\mstr$, then $\fsecq$ holds in $\mstr$.\label{thm:fgecq=>fsecq}
        \item If $\fspecq$ holds in $\mstr$ then $\fsecq$ also holds in it.\label{thm:fspecq=>fsecq}  
        \item If $\fpfecq$ holds in $\mstr$, then $\fsecq$ holds in $\mstr$.\label{thm:fpfecq=>secq}        
    \end{enumerate}
\end{thm}

The diagram depicting the interconnections between these principles of explosion is similar to Figure~\ref{fig:sat-exp}:

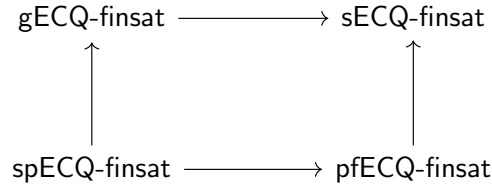
\begin{figure}[ht]
    \centering
\[\begin{tikzcd}
	{\mathsf{gECQ\text{-}finsat}} && {\mathsf{sECQ\text{-}finsat}} \\
	\\
	{\mathsf{spECQ\text{-}finsat}} && {\mathsf{pfECQ\text{-}finsat}}
	\arrow[from=1-1, to=1-3]
	\arrow[from=3-1, to=1-1]
	\arrow[from=3-1, to=3-3]
	\arrow[from=3-3, to=1-3]
\end{tikzcd}\]
    \caption{Semantic explosion principles - the $\mathsf{sat}$ variants}
    \label{fig:finsat-exp}
\end{figure}
The intended interpretation of the figure is the same as Figure \ref{fig:sat-exp}, namely, given an $\amst$ if it satisfies the principle of explosion at the tail of an arrow, it also satisfies the principle of explosion at the tip of the same arrow.

We now turn to showing that no other implication among these principles of explosion holds.

\begin{exa}[$\fsecq\centernot{\implies}\fgecq$]{\label{exm:fsecq≠>fgecq}}
    Let $\mstr=(\NN,\models, \mathcal{P}(\NN))$ be an $\amst$ where for all $\Gamma\cup\{m\}\subseteq \mathbb{N}$, $m\not\models \Gamma$ iff $|\Gamma|\ge 3$. In other words, a set is not satisfiable iff its cardinality is at least 3. 

    We note that $\mathfrak{M}$ satisfies $\fsecq$ because $\{0,1,2\}$ is not satisfiable. However, as every set of cardinality $\le 2$ is satisfiable, $\mathfrak{M}$ does not satisfy $\fgecq$.   
\end{exa}

\begin{rem}
\begin{enumerate}[label=(\roman*)]
    \item Note that $\fsecq$ does not imply $\fspecq$. Otherwise, as $\fspecq$ implies $\fgecq$, it would follow that $\fsecq$ implies $\fgecq$ $-$ contradicting the above example. 
    \item Let us consider the same $\amst$ as above and choose $\Gamma\subsetneq \NN$. 
    
    If $\Gamma=\emptyset$ then there exists a set, say $\{0,1,2\}\subsetneq \NN$ containing $\Gamma$ that is not satisfiable, or equivalently, not finitely satisfiable. If $\Gamma$ is nonempty and finite then $\Gamma\cup\{\max(\Gamma)+1,\max(\Gamma)+2,\max(\Gamma)+3\}$ - where $\max(\Gamma)$ denotes the maximum of $\Gamma$ - is such that $\Gamma\cup\{\max(\Gamma)+1,\max(\Gamma)+2,\max(\Gamma)+3\}\subsetneq \NN$. Since the cardinality of $\Gamma\cup\{\max(\Gamma)+1,\max(\Gamma)+2,\max(\Gamma)+3\}$ is at least 3, it is not satisfiable, and hence, not finitely satisfiable. Finally, if $\Gamma$ is infinite then choose any three elements belonging to $\Gamma$, say $\{m,n,k\}$. Since $\{m,n,k\}$ is not satisfiable, $\Gamma$ itself is not finitely satisfiable. Thus, for all $\Gamma\subsetneq \NN$ there exists $\Delta\subsetneq\NN$ such that $\Gamma\subseteq \Delta$ and $\Delta$ is not finitely satisfiable. This, howvever, implies that $\mstr$ satisfies $\fpfecq$ even though it does not satisfy $\fgecq$ as delineated in the above example. 
    
    Thus, $\fpfecq$ does not imply $\fgecq$.
    \item Since $\fspecq$ implies $\fgecq$ by Theorem~\ref{thm:finsat_imp}(i), while $\fpfecq$ does not imply $\fgecq$, as shown in (ii), we can conclude that $\fpfecq$ does not imply $\fspecq$.
\end{enumerate}    
\end{rem}
\begin{exa}[$\fgecq\centernot{\implies}\fpfecq$]{\label{exm:fgecq≠>fpfecq}}
    Consider the $\amst$ $\mstr=(\NN,\models, \mathcal{P}(\lang))$ where for all $\Gamma\cup\{m\}\subseteq \NN$, $m\models \Gamma$ iff $\Gamma\ne\{0\}$. In other words, $\{0\}$ is the unique unsatisfiable set. In particular, this implies that $\NN\setminus\{0\}$ is finitely satisfiable. Thus, there exists no $\Delta\subsetneq \NN$ such that $\NN\setminus\{0\}\subseteq\Delta\subsetneq \NN$ and $\Delta$ is not finitely satisfiable. Consequently, $\mstr$ does not satisfy $\fpfecq$. However, since for all $n\in \NN$, $\{n,0\}$ is not finitely satisfiable, $\mstr$ satisfies $\fgecq$. 
\end{exa}
\begin{rem}
    \begin{enumerate}[label=(\roman*)]
        \item Since by Theorem \ref{thm:finsat_imp}, $\fgecq$ implies $\fsecq$, from the above example it follows that $\fsecq$ does not imply $\fpfecq$.
        \item We note that $\fgecq$ does not imply $\fspecq$ either. Otherwise, as $\fspecq$ implies $\fpfecq$ (by Theorem \ref{thm:finsat_imp}(ii)), it would follow that $\fgecq$ implies $\fpfecq$ - contradicting the above example.       
    \end{enumerate}   
\end{rem}
\subsubsection{Characterization theorems}

We now turn to characterization results for the four principles of explosion in this subsection. Before going into the formal details of the theorems, we would like to point out that, unlike satisfiability, finite satisfiability is a `hereditary property', i.e., if a set is finitely satisfiable, every subset of the same is also finite satisfiable. This property turns out to be of crucial importance in the characterization theorems (cf. Theorems \ref{thm:char_secq} and \ref{thm:char_fsecq}).

\begin{thm}[Characterization of $\fsecq$]\label{thm:char_fsecq} 
Suppose $\mstr=(\mathbf{M},\models,\pow(\lang))$ is an $\amst$. The following conditions are equivalent.
\begin{enumerate}[label=(\roman*)]
    \item $\fsecq$ holds in $\mstr$.
    \item $\lang$ is not finitely satisfiable.
    \item For all $\alpha\in\lang$, there exists a finite $\Gamma\subseteq\lang$ such that $\Gamma\cup\{\alpha\}\subsetneq\lang$ and $\Gamma\cup\{\alpha\}$ is not finitely satisfiable.
    \item For all $\alpha\in\lang$, there exists $\beta\in\lang\setminus\{\alpha\}$ such that $\lang\setminus\{\beta\}$ is not finitely satisfiable.
\end{enumerate}
\end{thm}

\begin{proof}
   \underline{(i) $\implies$ (ii)}: Suppose $\fsecq$ holds in $\mstr$ and let $\alpha\in \lang$. Then, there exists $\Gamma\subsetneq\lang$ such that $\Gamma\cup\{\alpha\}$ is not finitely satisfiable. So, there exists a finite $\Gamma_0\subseteq \Gamma\cup\{\alpha\}$ that is not satisfiable. Thus, $\lang$ is not finitely satisfiable.  

    \underline{(ii) $\implies$ (iii)}: Suppose $\lang$ is not finitely satisfiable. So, there exists a finite $\Gamma\subseteq\lang$ that is not satisfiable. Let $\alpha\in\lang$. Since $\Gamma$ is a finite subset of $\Gamma\cup\{\alpha\}$ and is not satisfiable, $\Gamma\cup\{\alpha\}$ is not finitely satisfiable. Moreover, $\Gamma\cup\{\alpha\}$ is finite, while $\lang$ is infinite. So, $\Gamma\cup\{\alpha\}\subsetneq\lang$.

    \underline{(iii) $\implies$ (iv)}: Suppose the statement (iii) holds. Let $\alpha\in\lang$. Then, by the assumed condition, there exists a finite $\Gamma\subseteq\lang$ such that $\Gamma\cup\{\alpha\}\subsetneq\lang$ and $\Gamma\cup\{\alpha\}$ is not finitely satisfiable. Since $\Gamma\cup\{\alpha\}\subsetneq\lang$, there exists $\beta\in\lang\setminus(\Gamma\cup\{\alpha\})$. Then, $\beta\in\lang\setminus\{\alpha\}$ and $\Gamma\cup\{\alpha\}\subseteq\lang\setminus\{\beta\}$. Thus, as $\Gamma\cup\{\alpha\}$ is not finitely satisfiable, $\lang\setminus\{\beta\}$ is not finitely satisfiable either.

    \underline{(iv) $\implies$ (i)}: Suppose the statement (iv) holds and let $\alpha\in\lang$. Then, there exists $\beta\in\lang\setminus\{\alpha\}$ such that $\lang\setminus\{\beta\}$ is not finitely satisfiable. Now, $\alpha\in\lang\setminus\{\beta\}$ and $\lang\setminus\beta\subsetneq\lang$. Thus, $\fsecq$ holds in $\mstr$. 
\end{proof}

\begin{rem}
    By the above theorem, $\fsecq$ holds in $\mstr$ iff, for all $\alpha\in\lang$, there exists $\beta\in\lang\setminus\{\alpha\}$ such that $\lang\setminus\{\beta\}$ is not satisfiable relative to $\mathcal{K}=\{\Delta\subseteq\lang\setminus\{\beta\}\mid\,\Delta\hbox{ is finite}\}$ (cf. Theorem \ref{thm:char_secq}). 
\end{rem}

\begin{cor}
    Suppose $\mstr=(\mathbf{M},\models,\pow(\lang))$ is a normal $\amst$ and $\mathcal{S}=(\lang,\vdash)$ is the logical structure induced by $\mstr$. If $\mathcal{S}$ is finitary and $\fsecq$ holds in $\mstr$, then $\mstr$ is compact.
\end{cor}

\begin{proof}
    Since $\mstr$ satisfies $\fsecq$, by Theorem \ref{thm:char_fsecq}, $\lang$ is not finitely satisfiable, i.e., there exists a finite unsatisfiable set. So, by Corollary \ref{cor:finunsat=>compact}, it follows that $\mstr$ is compact.
\end{proof}

\begin{thm}[Characterization of $\fpfecq$]\label{thm:char_fpfecq}
    Suppose $\mstr=(\mathbf{M},\models,\pow(\lang))$ is an $\amst$. $\fpfecq$ holds in $\mstr$ iff, for all $\alpha\in\lang$, $\lang\setminus\{\alpha\}$ is not finitely satisfiable.
\end{thm}

\begin{proof}
    Suppose $\fpfecq$ holds in $\mstr$ and $\alpha\in\lang$. Then, as $\lang\setminus\{\alpha\}\subsetneq\lang$, there exists $\Delta\subsetneq\lang$ such that $\lang\setminus\{\alpha\}\subseteq\Delta$ and $\Delta$ is not finitely satisfiable. Clearly, $\alpha\notin\Delta$, since otherwise, $\Delta=\lang$, a contradiction. Thus, $\Delta=\lang\setminus\{\alpha\}$. Hence, $\lang\setminus\{\alpha\}$ is not finitely satisfiable.

    Conversely, suppose $\lang\setminus\{\alpha\}$ is not finitely satisfiable for all $\alpha\in\lang$. Moreover, suppose $\fpfecq$ fails in $\mstr$. Then, there exists $\Gamma\subsetneq\lang$ such that, for all $\Delta\subsetneq\lang$, if $\Gamma\subseteq\Delta$, then $\Delta$ is finitely satisfiable. Now, as $\Gamma\subsetneq\lang$, there exists $\alpha\in\lang\setminus\Gamma$. Then, $\Gamma\subseteq\lang\setminus\{\alpha\}\subsetneq\lang$. So, $\lang\setminus\{\alpha\}$ is finitely satisfiable, which contradicts our assumption. Hence, $\fpfecq$ holds in $\mstr$. 
\end{proof}

Let $\mstr=(\mathbf{M},\models, \mathcal{P}(\lang))$ be an $\amst$. If $\mstr$ satisfies $\fsecq$ then by Theorem \ref{thm:char_fsecq}, there exists a finite subset of $\lang$ that is not satisfiable. The following result says that if, moreover, $\mstr$ does not satisfy $\fpfecq$, then the finite unsatisfiable set(s) cannot be mutually disjoint.

\begin{thm}\label{thm:fin+disj+nsat=>fpfecq}
    Suppose $\mstr=(\mathbf{M},\models,\pow(\lang))$ is an $\amst$. Moreover, suppose there exist $\Sigma,\Delta\subseteq\lang$ such that $\Sigma\cup\Delta$ is finite, $\Sigma\cap\Delta=\emptyset$, and $\Sigma,\Delta$ are not satisfiable. Then, $\fpfecq$ holds in $\mstr$.
\end{thm}

\begin{proof}
    Let $\alpha\in\lang$. Since $\Sigma\cap\Delta=\emptyset$, the following are the only possible cases.

    \textbf{Case 1:} $\alpha\notin\Sigma$ but $\alpha\in\Delta$

    In this case $\Sigma\subseteq\lang\setminus\{\alpha\}$. Since $\Sigma$ is not satisfiable and finite, this implies that $\lang\setminus\{\alpha\}$ is not finitely satisfiable.

    \textbf{Case 2:} $\alpha\notin\Delta$ but $\alpha\in\Sigma$

    In this case $\Delta\subseteq\lang\setminus\{\alpha\}$. Since $\Delta$ is finite and not satisfiable, this implies again that $\lang\setminus\{\alpha\}$ is not finitely satisfiable.

    \textbf{Case 3:} $\alpha\notin\Sigma\cup\Delta$

    In this case $\Sigma\subseteq\Sigma\cup\Delta\subseteq\lang\setminus\{\alpha\}$. Then, again as $\Sigma$ is finite and not satisfiable, $\lang\setminus\{\alpha\}$ is not finitely satisfiable.

    Thus, $\lang\setminus\{\alpha\}$ is not finitely satisfiable for all $\alpha\in\lang$. Hence, by Theorem \ref{thm:char_fpfecq}, $\fpfecq$ holds in $\mstr$.
\end{proof}

\begin{cor}\label{cor:disj+nfsat=>fpfecq}
    Suppose $\mstr=(\mathbf{M},\models,\pow(\lang))$ is an $\amst$. Moreover, suppose there exist $\Sigma,\Delta\subseteq\lang$ such that $\Sigma\cap\Delta=\emptyset$, and $\Sigma,\Delta$ are not finitely satisfiable. Then, $\fpfecq$ holds in $\mstr$.
\end{cor}

\begin{proof}
    Since $\Sigma$ and $\Delta$ are not finitely satisfiable, there exist a finite $\Gamma^\prime\subseteq\Gamma$ and a $\Delta^\prime\subseteq\Delta$ such that $\Gamma^\prime,\Delta^\prime$ are not satisfiable. Moreover, since $\Sigma\cap\Delta=\emptyset$, $\Sigma^\prime\cap\Delta^\prime=\emptyset$ as well. Hence, by Theorem \ref{thm:fin+disj+nsat=>fpfecq}, $\fpfecq$ holds in $\mstr$.
\end{proof}

\begin{thm}[Characterization of $\fspecq$]\label{thm:char_fspecq}
    Suppose $\mstr=(\mathbf{M},\models,\pow(\lang))$ is an $\amst$. $\fspecq$ holds in $\mstr$ iff $\fpfecq$ holds in $\mstr$ and there exists $\varphi\in\lang$ such that $\{\varphi\}$ is not finitely satisfiable.
\end{thm}

\begin{proof}
    Suppose $\fspecq$ holds in $\mstr$. Let $\Gamma\subsetneq\lang$. Then, as $\fspecq$ holds in $\mstr$, there exists $\alpha\in\lang$ such that $\Gamma\cup\{\alpha\}\subsetneq\lang$ and $\Gamma\cup\{\alpha\}$ is not finitely satisfiable. Since $\Gamma\subseteq\Gamma\cup\{\alpha\}$, this implies that $\fpfecq$ holds in $\mstr$. Now, $\emptyset\subsetneq\lang$, and so, again by $\fspecq$, there exists $\varphi\in\lang$ such that $\emptyset\cup\{\varphi\}=\{\varphi\}\subsetneq\lang$ (since $\lang$ is infinite) and $\{\varphi\}$ is not finitely satisfiable. 

    Conversely, suppose $\fpfecq$ holds in $\mstr$ and there exists $\varphi\in\lang$ such that $\{\varphi\}$ is not finitely satisfiable. Let $\Gamma\subsetneq\lang$. If $\Gamma=\emptyset$, then $\Gamma\cup\{\varphi\}=\{\varphi\}\subsetneq\lang$ and $\Gamma\cup\{\varphi\}$ is not finitely satisfiable. On the other hand, if $\Gamma\neq\emptyset$, then the following two cases arise.

    \textbf{Case 1:} $\Gamma=\lang\setminus\{\beta\}$ for some $\beta\in\lang$. Then, by Theorem \ref{thm:char_fpfecq}, $\Gamma$ is not finitely satisfiable. So, for any $\gamma\in\Gamma$, $\Gamma\cup\{\gamma\}=\Gamma\subsetneq\lang$ and is not finitely satisfiable.

    \textbf{Case 2:} $\Gamma\neq\lang\setminus\{\beta\}$ for all $\beta\in\lang$. Then, $\Gamma\cup\{\varphi\}\subsetneq\lang$ and $\Gamma\cup\{\varphi\}$ is not finitely satisfiable as $\{\varphi\}$ is a subset of it that is not finitely satisfiable. 

    Thus, for all $\Gamma\subsetneq\lang$, there exists $\alpha\in\lang$ such that $\Gamma\cup\{\alpha\}\subsetneq\lang$ and $\Gamma\cup\{\alpha\}$ is not finitely satisfiable.
\end{proof}

A characterization for $\specq$ was obtained in Theorem \ref{thm:char_specq}. We end this section with the following theorems, which gives alternative partial characterizations for $\specq$ via finite satisfiability.

\begin{thm}\label{thm:specq/fspecq}
    Suppose $\mstr=(\mathbf{M},\models,\pow(\lang))$ is an $\amst$. If $\specq$ holds in $\mstr$, then either $\fspecq$ holds in it, or there exists a unique $\varphi\in\lang$ such that $\{\varphi\}$ is not satisfiable and for all finite $\Gamma\subseteq\lang$, $\Gamma$ is not satisfiable iff $\varphi\in\Gamma$.
\end{thm}

\begin{proof}
    Suppose $\specq$ holds in $\mstr$. Then, as $\emptyset\subsetneq\lang$, there exists $\varphi\in\lang$ such that $\emptyset\cup\{\varphi\}=\{\varphi\}\subsetneq\lang$ and $\{\varphi\}$ is not satisfiable. Now, $\lang\setminus\{\varphi\}$ is either finitely satisfiable or not.

    \textbf{Case 1:} $\lang\setminus\{\varphi\}$ is finitely satisfiable. Then, for all $\alpha\in\lang\setminus\{\varphi\}$, $\{\alpha\}$ is satisfiable. 
    
    We now claim that for any finite $\Gamma\subseteq\lang$, $\Gamma$ is not satisfiable iff $\varphi\in\Gamma$. Let $\Gamma$ be a finite subset of $\lang$. 
    
    Suppose $\Gamma$ is not satisfiable. Then, as $\lang\setminus\{\varphi\}$ is finitely satisfiable, $\Gamma\not\subseteq\lang\setminus\{\varphi\}$. So, $\varphi\in\Gamma$. 
    
    Conversely, suppose $\varphi\in\Gamma$. Let $\Sigma=\Gamma\setminus\{\varphi\}$. Then, $\Sigma$ is a finite subset of $\lang\setminus\{\varphi\}$, and hence, is satisfiable. Now, as $\specq$ holds in $\mstr$, there exists $\gamma\in\lang$ such that $\Sigma\cup\{\gamma\}\subsetneq\lang$ (since $\Sigma\cup\{\gamma\}$ is finite, while $\lang$ is infinite) and $\Sigma\cup\{\gamma\}$ is not satisfiable. Then, as $\lang\setminus\{\varphi\}$ is finitely satisfiable, $\Sigma\cup\{\gamma\}\not\subseteq\lang\setminus\{\varphi\}$, which implies that $\varphi\in\Sigma\cup\{\gamma\}$. Now, $\varphi\notin\Sigma=\Gamma\setminus\{\varphi\}$. Hence, $\gamma=\varphi$. So, $\Gamma=\Sigma\cup\{\varphi\}=\Sigma\cup\{\gamma\}$ is not satisfiable.

    Moreover, if $\varphi^\prime\in\lang$ be such that $\{\varphi^\prime\}$ is not satisfiable, then by the above arguments, $\varphi\in\{\varphi^\prime\}$, since $\{\varphi^\prime\}$ is a finite subset of $\lang$ that is not satisfiable. So, $\varphi=\varphi^\prime$, which implies that $\varphi$ is the unique element of $\lang$ such that $\{\varphi\}$ is not satisfiable. 

    \textbf{Case 2:} $\lang\setminus\{\varphi\}$ is not finitely satisfiable. Then, there exists a finite $\Gamma\subseteq\lang\setminus\{\varphi\}$ that is not satisfiable. Now, $\{\varphi\}$ is also not satisfiable. Moreover, $\Gamma\cup\{\varphi\}$ is finite and $\Gamma\cap\{\varphi\}=\emptyset$. So, by Theorem \ref{thm:fin+disj+nsat=>fpfecq}, $\fpfecq$ holds in $\mstr$. Now, $\{\varphi\}$ is a finite subset of itself and is not satisfiable. Thus, $\{\varphi\}$ is not finitely satisfiable. Hence, by Theorem \ref{thm:char_fspecq}, $\fspecq$ holds in $\mstr$.
\end{proof}

\begin{thm}\label{thm:char_specq=>}
    Suppose $\mstr=(\mathbf{M},\models,\pow(\lang))$ is an $\amst$. If $\specq$ holds in $\mstr$, then $\lang\setminus\{\alpha\}$ is finitely satisfiable for at most one $\alpha\in\lang$.
\end{thm}

\begin{proof}
    Suppose $\specq$ holds in $\mstr$. Then, since $\emptyset\subsetneq\lang$, there exists $\alpha\in\lang$ such that $\emptyset\cup\{\alpha\}=\{\alpha\}\subsetneq\lang$ is not satisfiable. Now, $\lang\setminus\{\alpha\}$ is either finitely satisfiable or not.

    If $\lang\setminus\{\alpha\}$ is not finitely satisfiable, then there is a finite $\Delta\subseteq\lang\setminus\{\alpha\}$ that is not satisfiable. So, $\Delta$ and $\{\alpha\}$ are two finite sets such that $\Delta\cap\{\alpha\}=\emptyset$, by Theorem \ref{thm:fin+disj+nsat=>fpfecq}, $\fpfecq$ holds in $\mstr$. Then, by Theorem \ref{thm:char_fpfecq}, $\lang\setminus\{\beta\}$ is not finitely satisfiable for all $\beta\in\lang$.

    On the other hand, if $\lang\setminus\{\alpha\}$ is finitely satisfiable, then for any $\beta\neq\alpha$, since $\{\alpha\}$ is a finite subset of $\lang\setminus\{\beta\}$ and is not satisfiable, $\lang\setminus\{\beta\}$ is not finitely satisfiable. So, $\lang\setminus\{\alpha\}$ is the only finitely satisfiable set of this form.
\end{proof}

\subsection{Semantic explosion principles - the \texorpdfstring{$\mathsf{sat}$}{sat} variants vis-à-vis
 the \texorpdfstring{$\mathsf{finsat}$}{finsat} variants}

In Subsections \ref{subsec:sat} and \ref{subsec:finsat} we have gradually introduced the $\mathsf{sat}$- as well as the $\mathsf{finsat}$-variants of the semantic explosion principles. Interconnections between them gave rise to Figures \ref{fig:sat-exp} and \ref{fig:finsat-exp} respectively. We now investigate the interconnections between all these principles of explosion. The resulting diagram is Figure \ref{fig:sem-exp}).

In addition to the implications between the explosion principles as recorded in Theorems \ref{thm:sat_imp} and \ref{thm:finsat_imp}, we also have the following.

\begin{thm}[Implications Between $\mathsf{sat}$- vis-à-vis $\mathsf{finsat}$-variants]{\label{thm:sem_imp}}
    Suppose $\mstr=(\mathbf{M},\models,\pow(\lang))$ is an $\amst$. Then, the following statements hold.
    \begin{enumerate}[label=(\roman*)]
        \item If $\specq$ holds in $\mstr$, then $\fgecq$ holds in it.\label{thm:specq=>fgecq}
        \item If $\gecq$ holds in $\mstr$, then $\fgecq$ holds in $\mstr$.\label{thm:gecq=>fgecq}
        \item If $\specq$ holds in $\mstr$ then $\fsecq$ also holds in it.\label{thm:specq=>fsecq}         
        \item If $\gecq$ holds in $\mstr$ then $\fsecq$ also holds in it.\label{thm:gecq=>fsecq}           
    \end{enumerate}
\end{thm}
\begin{proof}
    \begin{enumerate}[label=(\roman*)]
        \item Since $\specq$ holds in $\mstr$ then by Theorem \ref{thm:sat_imp}(i) $\gecq$ holds in $\mstr$. Now, let $\alpha\in \lang$. By $\gecq$, there exists $\beta\in \lang$ such that $\{\alpha,\beta\}$ is not satisfiable. Since $\{\alpha,\beta\}$ is finite, this implies that $\{\alpha,\beta\}$ is not finitely satisfiable. Thus, $\gecq$ holds in $\mstr$.
        \item Follows from (i).
        \item Since $\specq$ holds in $\mstr$, by (i), $\fgecq$ holds in it. Then, by Theorem \ref{thm:finsat_imp}(iv), $\fsecq$ holds in $\mstr$.
        \item Since $\gecq$ holds in $\mstr$, by (ii), $\fgecq$ holds in it. Then, by Theorem \ref{thm:finsat_imp}(iii), $\fsecq$ holds in $\mstr$.
    \end{enumerate}
\end{proof}

\begin{rem}
    Since by the above theorem $\gecq$ implies $\fgecq$, from Example~\ref{exm:gecq/secq≠>pfecq/specq} it follows that $\fgecq$ does not imply $\pfecq$ or $\specq$.    
\end{rem}

We now proceed to show, via examples, that no other implication among these explosion principles is possible.

\begin{exa}[$\fspecq\centernot\implies\secq$]\label{exm:fspecq≠>gecq/secq}
    Let $\mstr=(\NN,\models,\pow(\NN))$ be an $\amst$, where $\models\,\subseteq\NN\times\pow(\NN)$ is defined as follows. For all $m\in\NN$ and $\Gamma\subseteq\NN$, $m\models\Gamma$ iff $\Gamma\neq\{n\}$ for some $n\neq0$. 

    Let $\Gamma\subsetneq\NN$. If $\Gamma=\emptyset$, then $\Gamma\cup\{1\}=\{1\}\subsetneq\NN$ is not satisfiable and hence, not finitely satisfiable. Next, if $\Gamma\neq\emptyset$ and there exists $n\in\Gamma$ such that $n\neq0$, then $\Gamma\cup\{n\}=\Gamma\subsetneq\NN$ and since $\{n\}$ is not satisfiable, $\Gamma\cup\{n\}$ is not finitely satisfiable. Finally, if $\Gamma=\{0\}$, then $\Gamma\cup\{1\}=\{0,1\}\subsetneq\NN$ and is not finitely satisfiable as $\{1\}\subseteq\{0,1\}$ is not satisfiable. Thus, $\fspecq$ holds in $\mstr$.

    Now, %for all $n\in\NN$, $\{0,n\}$ is satisfiable. Thus, $\gecq$ fails in $\mstr$. Moreover, 
    for all $\Gamma\subseteq\lang$, $\Gamma\cup\{0\}$ is satisfiable. Thus, $\secq$ fails in $\mstr$ as well.
\end{exa}

\begin{rem}
    Since $\fspecq$ implies $\fgecq,\fpfecq$ and $\fsecq$, the above example also shows that none of these imply $\secq$. Consequently, as neither $\fsecq$ nor $\fpfecq$ imply $\secq$, $\fsecq$ or $\fpfecq$ does not imply $\specq,\gecq$ or $\pfecq$.
\end{rem}

\begin{exa}[$\pfecq\centernot\implies\fsecq$]\label{exm:pfecq≠>fsecq}
    Let $\mstr=(\NN,\models,\pow(\NN))$ be an $\amst$, where $\models\,\subseteq\NN\times\pow(\NN)$ is defined as follows. For any $\Gamma\cup\{m\}\subseteq\NN$, $m\not\models\Gamma$ iff either $\Gamma=\NN\setminus\{n\}$ for some $n\in\NN$ or $m\in\Gamma$.

    By definition of $\mstr$, $\NN\setminus\{n\}$ is not satisfiable for all $n\in\NN$. So, by Theorem \ref{thm:char_pfecq}, $\pfecq$ holds in $\mstr$. 

    Now, suppose $\Gamma\subseteq\NN$ is finite. Then, $\Gamma$ cannot be of the form $\NN\setminus\{n\}$ for some $n\in\NN$ and there exists $m\in\NN\setminus\Gamma$. So, again by definition of $\mstr$, $m\models\Gamma$. Thus, every finite subset of $\NN$ is satisfiable. Hence, every subset of $\NN$ is finitely satisfiable. Thus, $\fsecq$ does not hold in $\mstr$.
\end{exa}

\begin{rem}
Since $\fspecq,\fgecq,\fpfecq$ implies $\fsecq$, the above example also shows that $\pfecq$ does not imply any  of these. Consequently, as $\pfecq$ implies $\secq$, it follows that $\secq$ does not imply $\fspecq,\fgecq$ or $\fpfecq$.  
\end{rem}

\begin{exa}[$\specq\centernot\implies\fpfecq$]\label{exm:pfecq/specq/gecq/secq/fsecq≠>fpfecq}
    Let $\mstr=(\NN,\models,\pow(\NN))$ be an $\amst$, where $\models\,\subseteq\NN\times\pow(\NN)$ is defined as follows. For any $\Gamma\cup\{m\}\subseteq\NN$, $m\models\Gamma$ iff either $\Gamma=\emptyset$ or $\Gamma\neq\emptyset$ is finite with $0\notin\Gamma$ and $m\in\Gamma$.

    Let $\Gamma\subsetneq\NN$. If $\Gamma$ is infinite, then $\Gamma$ is not satisfiable. So, for any $m\in\Gamma$, $\Gamma\cup\{m\}=\Gamma\subsetneq\NN$ is not satisfiable. On the other hand, if $\Gamma$ is finite, then $\Gamma\cup\{0\}$ is finite and $\Gamma\cup\{0\}\subsetneq\NN$. Moreover, $\Gamma\cup\{0\}$ is not satisfiable. Thus, $\specq$ holds in $\mstr$.

   Now, every finite subset of $\NN\setminus\{0\}$ is satisfiable. So, by Theorem \ref{thm:char_fpfecq}, $\fpfecq$ does not hold in $\mstr$.
\end{exa}

\begin{rem}
Since $\specq$ implies $\gecq,\fgecq,\pfecq,\secq$ and $\fsecq$ the above example also show that none of these principles of explosion imply $\fpfecq$. Hence, as $\fspecq$ implies $\fpfecq$, it also follows that $\fspecq$ is not implied by any one of $\gecq$, $\fgecq$, $\pfecq$, $\secq$, and $\fsecq$.
\end{rem}

\begin{exa}[$\fsecq/\pfecq/\fpfecq\centernot\implies\fgecq$]\label{exm:fsecq/pfecq/fpfecq≠>fgecq}
    Let $\mstr=(\NN,\models,\pow(\NN))$ be an $\amst$, where $\models\,\subseteq\NN\times\pow(\NN)$ is defined as follows. For any $\Gamma\cup\{m\}\subseteq\NN$, $m\not\models\Gamma$ iff $\lvert\Gamma\rvert\ge3$. Then, as for any $n\in\NN$, $\lvert\NN\setminus\{n\}\rvert\ge3$, $\NN\setminus\{n\}$ is not satisfiable for all $n\in\NN$. So, by Theorem \ref{thm:char_pfecq}, $\pfecq$ holds in $\mstr$.

    For any $n\in\NN$, $\{n+1,n+2,n+3\}\subseteq\NN$ and is not satisfiable. Thus, $\NN\setminus\{n\}$ is not finitely satisfiable for all $n\in\NN$. So, by Theorem \ref{thm:char_fpfecq}, $\fpfecq$ holds in $\mstr$.

    For any $n\in\NN$, $\{n+1,n+2\}\cup\{n\}\subsetneq\NN$ and $\{n+1,n+2\}\cup\{n\}=\{n,n+1,n+2\}$ is not satisfiable, and hence, also not finitely satisfiable. Thus, $\secq$ and $\fsecq$ hold in $\mstr$.

    However, as every set with at most two elements is satisfiable, and hence finitely satisfiable, $\fgecq$ does not hold in $\mstr$.
\end{exa}

Our final example shows that $\fspecq$ does not imply $\secq$. Now, as $\fspecq$ implies $\fpfecq$, the same example also shows that $\fpfecq$ does not imply $\secq$.

\begin{exa}[$\fspecq/\fpfecq\centernot\implies\secq$]\label{exm:fspecq≠>secq}
    Let $\mstr=(\NN,\models,\pow(\NN))$ be an $\amst$, where $\models\,\subseteq\NN\times\pow(\NN)$ is defined as follows. For any $\Gamma\cup\{m\}\subseteq\NN$, $m\models\Gamma$ iff $\Gamma\neq\emptyset$. Thus, every non-empty set is satisfiable but, as $\emptyset$ is not satisfiable, for any $\Gamma\subseteq\NN$, $\Gamma$ is not finitely satisfiable. Thus, for any $n\in\NN$, $\NN\setminus\{n\}\neq\emptyset$, and hence, is not finitely satisfiable. So, by Theorem \ref{thm:char_fpfecq}, $\fpfecq$ holds in $\mstr$. Moreover, for any $n\in\NN$, $\{n\}$ is not finitely satisfiable. Hence, by Theorem \ref{thm:char_fspecq}, $\fspecq$ holds in $\mstr$.

    Now, as every non-empty set is satisfiable, there exists $n\in\NN$ (in fact, for any $n\in\NN$) such that, for all $\Gamma\subseteq\NN$, $\Gamma\cup\{n\}$ is satisfiable. This implies that $\secq$ does not hold in $\mstr$.
\end{exa}

The Figure \ref{fig:sem-exp} shows all the interconnections between the principles discussed in this article.

\begin{figure}[H]
\centering
\[\begin{tikzcd}[ampersand replacement=\&,cramped]
	{\mathsf{pfECQ}{\text{-}}{\mathsf{sat}}} \&\&\&\& {\mathsf{pfECQ}{\text{-}}{\mathsf{finsat}}} \\
	{\mathsf{spECQ}{\text{-}}{\mathsf{sat}}} \&\&\&\& {\mathsf{spECQ}{\text{-}}{\mathsf{finsat}}} \\
	{\mathsf{gECQ}{\text{-}}{\mathsf{sat}}} \&\&\&\& {\mathsf{gECQ}{\text{-}}{\mathsf{finsat}}} \\
	{\mathsf{sECQ}{\text{-}}{\mathsf{sat}}} \&\&\&\& {\mathsf{sECQ}{\text{-}}{\mathsf{finsat}}}
	\arrow["{\boxed{\text{Compactness}}}"{description}, dashed, tail reversed, from=1-1, to=1-5]
	\arrow[curve={height=50pt}, from=1-1, to=4-1]
	\arrow[curve={height=-50pt}, from=1-5, to=4-5]
	\arrow[from=2-1, to=1-1]
	\arrow["{\boxed{\text{Compactness}}}"{description}, dashed, tail reversed, from=2-1, to=2-5]
	\arrow[from=2-1, to=3-1]
	\arrow[from=2-5, to=1-5]
	\arrow[from=2-5, to=3-5]
	\arrow[curve={height=-12pt}, from=3-1, to=3-5]
	\arrow[from=3-1, to=4-1]
	\arrow["{\boxed{\text{Compactness}}}"{description}, curve={height=-12pt}, dashed, from=3-5, to=3-1]
	\arrow[from=3-5, to=4-5]
	\arrow["{\boxed{\text{Compactness}}}"{description}, dashed, tail reversed, from=4-1, to=4-5]
\end{tikzcd}\]
\caption{Semantic Explosion}
    \label{fig:sem-exp}
\end{figure}
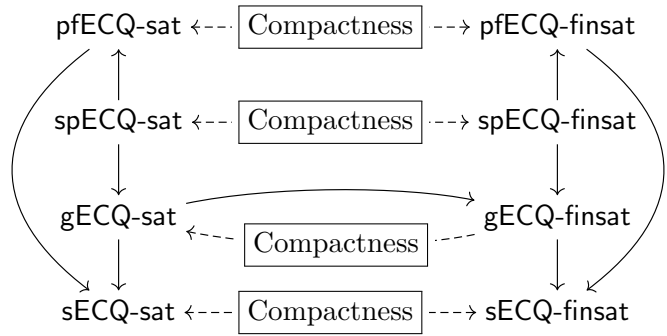

\section{Epilogue}

This article is a continuation of our study of generalized explosion principles through the lens of universal logic that was started in \cite{BasuRoy2022} and later continued in \cite{BasuRoy2024}. In this article, we have explored explosion principles from a semantic point of view using the concept of abstract model structures, introduced in \cite{RoyBasuChakraborty2025}. The characterization theorems proved in Subsection \ref{subsec:sat}, along with Theorem \ref{thm:logstr->amst}, provide characterization theorems for the principles of explosion discussed in \cite{BasuRoy2024}.

In this article, the semantic explosion principles have been obtained by replacing syntactic explosion with a specific interpretation of semantic explosion, which, for us, was identified with unsatisfiability. We have also observed that if the claim to compactness is given up, then we find another set of variants of these semantic explosion principles in terms of finite unsatisfiability.

However, there are other ways of interpreting semantic explosion. For example, we say that a logical structure $(\lang,\vdash)$ satisfies $\text{gECQ}$ if for all $\alpha\in \lang$ there exists $\beta\in \lang$ such that $\{\alpha,\beta\}\vdash\gamma$ for all $\gamma\in \lang$ (see \cite[Section 2]{BasuRoy2024}). If $\mstr=(\mathbf{M},\models,\mathcal{P}(\lang))$ is an $\amst$ such that $\vdash\subseteq\vdash_{\mstr}$, then the above statement is equivalent to the following: for all $\alpha\in \lang$, there exists $\beta\in \lang$ such that $\Mod(\{\alpha,\beta\})\subseteq\displaystyle\bigcap_{\gamma\in \lang}\Mod(\{\gamma\})$. One can thus interpret semantic explosion as $\vdash_{\mstr}$-triviality and obtain the corresponding semantic explosion principles.

An important set of explosion principles that has not been discussed in this article is the so-called `principles of partial explosion' introduced in \cite[Section 4]{BasuRoy2024}. Interconnections between its semantic variants and the ones introduced in this article remain an important direction for research.

The present study can be further extended to study relativized notions of satisfiability. More specifically, instead of replacing syntactic explosion with satisfiability/finite satisfiability, one can replace it with $\mathcal{K}$-satisfiability, introduced in Definition \ref{dfn:rel_sat}. By varying the set $\mathcal{K}$ in this definition, one can obtain different degrees of satisfiability. Such an investigation will also provide a unified perspective of the results proved in this paper because satisfiability and finite satisfiability can easily be seen to be special cases of $\mathcal{K}$-satisfiability. 

Apart from this, there are myriad directions in which we plan to extend the work on explosion principles within the general setup of universal logic. We leave these as future work.

\bibliographystyle{amsplain}
\bibliography{pracparacons}
\end{document}